\documentclass[11pt,reqno]{article}

\usepackage{amsfonts}
\usepackage{amscd}
\usepackage{color}
\usepackage{amsmath}
\usepackage{float}
\usepackage{amsthm}
\usepackage{amssymb}
\usepackage{mathrsfs}
\usepackage{amstext}
\usepackage{graphicx}
\usepackage{tikz}
\usepackage{verbatim}
\font\msbm=msbm10

\numberwithin{equation}{section}

\theoremstyle{plain}
\newtheorem{satz}{Theorem}[section]
\newtheorem{defi}[satz]{Definition}
\newtheorem{cor}[satz]{Corollary}
\newtheorem{lem}[satz]{Lemma}
\newtheorem{prop}[satz]{Proposition}
\newtheorem{rem}[satz]{Remark}

\newcommand{\mix}{{\rm mix}}
\newcommand{\re}{\ensuremath{\mathbb{R}}}\newcommand{\N}{\ensuremath{\mathbb{N}}}
\newcommand{\zz}{\ensuremath{\mathbb{Z}}}\newcommand{\C}{\ensuremath{\mathbb{C}}}
\newcommand{\T}{\ensuremath{\mathbb{T}^d}}\newcommand{\tor}{\ensuremath{\mathbb{T}}}
\newcommand{\com}{\ensuremath{\mathbb{C}}}
\newcommand{\Z}{{\ensuremath{\zz}^d}}
\newcommand{\Zp}{{\zz^d_+}}
\newcommand{\n}{\ensuremath{{\N}_0}}
\newcommand{\R}{\ensuremath{{\re}^d}}

\newcommand{\cs}{\ensuremath{\mathcal S}}
\newcommand{\cd}{\ensuremath{\mathcal D}}
\newcommand{\cl}{\ensuremath{\mathcal L}}
\newcommand{\cm}{\ensuremath{\mathcal M}}
\newcommand{\ca}{\ensuremath{\mathcal A}}
\newcommand{\cx}{\ensuremath{\mathcal X}}
\newcommand{\crp}{\ensuremath{\mathcal R}}
\newcommand{\ch}{\ensuremath{\mathcal H}}
\newcommand{\co}{\ensuremath{\mathcal O}}
\newcommand{\ce}{\ensuremath{\mathcal E}}
\newcommand{\ct}{\ensuremath{\mathcal T}}
\newcommand{\cn}{\ensuremath{\mathcal N}}
\newcommand{\cf}{\ensuremath{\mathcal F}}
\newcommand{\cfi}{\ensuremath{{\cf}^{-1}}}
\newcommand{\ci}{\ensuremath{\mathcal I}}
\newcommand{\cg}{\ensuremath{\mathcal G}}
\newcommand{\cv}{\ensuremath{\mathcal V}}
\newcommand{\cp}{\ensuremath{\mathcal P}}
\newcommand{\trigpol}{\mathcal{T}^{\ell}}
\newcommand{\trigpolk}{\mathcal{T}^{k}}
\newcommand{\kone}{\ensuremath{|k|_1}}
\newcommand{\kinf}{\ensuremath{|k|_{\infty}}}
\newcommand{\lone}{\ensuremath{|\ell|_1}}
\newcommand{\linf}{\ensuremath{|\ell|_{\infty}}}
\newcommand{\hiso}{\ensuremath{H^{\gamma}(\tor^d)}}
\newcommand{\hisobeta}{\ensuremath{H^{\beta}(\tor^d)}}
\newcommand{\ltwo}{\ensuremath{L_{2}(\tor^d)}}
\newcommand{\deltak}{\ensuremath{\|\delta_k(f)\|_2}}
\newcommand{\deltal}{\ensuremath{\|\delta_{\ell}(f)\|_2}}
\newcommand{\talpha}{\ensuremath{\tilde{\alpha}}}
\newcommand{\tbeta}{\ensuremath{\tilde{\beta}}}
\newcommand{\tk}{\ensuremath{\tilde{k}}}
\newcommand{\tgamma}{\ensuremath{\tilde{\gamma}}}
\newcommand{\ttalpha}{\ensuremath{\tilde{\tilde{\alpha}}}}
\newcommand{\ttbeta}{\ensuremath{\tilde{\tilde{\beta}}}}
\newcommand{\ttk}{\ensuremath{\tilde{\tilde{k}}}}
\newcommand{\ttgamma}{\ensuremath{\tilde{\tilde{\gamma}}}}
\newcommand{\fhab}{\ensuremath{\|f\|_{H^{\alpha,\beta}(\tor^d)}}}
\newcommand{\fmixed}{\ensuremath{\|f\|_{H^{\alpha}_{\mix}(\tor^d)}}}
\newcommand{\fmixedq}{\ensuremath{\|f\|^*_{H^{\alpha}_{\mix}(\tor^d)}}}
\newcommand{\fiso}{\ensuremath{\|f\|_{H^{\gamma}(\tor^d)}}}
\newcommand{\hmixed}{\ensuremath{H^{\alpha}_{\mix}(\tor^d)}}
\newcommand{\hmixedta}{\ensuremath{H^{\talpha}_{\mix}(\tor^d)}}
\newcommand{\hmixedtta}{\ensuremath{H^{\zeta}_{\mix}(\tor^d)}}
\newcommand{\hmixedga}{\ensuremath{H^{\gamma}_{\mix}(\tor^d)}}
\def\II{{\mathbb I}}
\def\CC{{\mathbb C}}
\def\ZZ{{\mathbb Z}}
\def\NN{{\mathbb N}}
\def\RR{{\mathbb R}}
\def\TTd{{\mathbb T}^d}
\newcommand{\fmixedv}{\ensuremath{\|f\|_{H^{\alpha}_{\mix}(\tor^d)}}}
\newcommand{\fmixedqv}{\ensuremath{\|f\|^*_{H^{\alpha}_{\mix}(\tor^d)}}}
\newcommand{\hmixedv}{\ensuremath{H^{\alpha}_{\mix}(\tor^d)}}
\newcommand{\hmixedtav}{\ensuremath{H^{\talpha\bf{1}}_{\mix}(\tor^d)}}
\newcommand{\hmixedgav}{\ensuremath{H^{\gamma}_{\mix}(\tor^d)}}
\newcommand{\hab}{\ensuremath{H^{\alpha,\beta}(\tor^d)}}

\definecolor{dgreen}{rgb}{0,.6,0}
\newcommand{\red}[1]{{\color{red}{#1}}}
\newcommand{\blue}[1]  {{\color{blue}{#1}}}
\newcommand{\dgreen}[1]  {{\color{dgreen}{#1}}}

\newcommand{\Span}{{\rm span \, }}
\newcommand{\rank}{{\rm rank \, }}
\newcommand{\supp}{{\rm supp \, }}
\newcommand{\sign}{{\rm sign \, }}
\newcommand{\sinc}{{\rm sinc}\,}
\newcommand{\inti}{\int_{-\infty}^\infty}
\newcommand{\sumk}{\sum_{k=-\infty}^\infty}
\newcommand{\summ}{\sum_{m=-\infty}^\infty}
\newcommand{\suml}{\sum_{\ell=-\infty}^\infty}
\newcommand{\sumj}{\sum_{j=0}^\infty}
\newcommand{\sumld}{{\sum_{\ell_1 = -\infty}^\infty \sum_{\ell_2 = - \infty}^\infty }}
\newcommand{\wt}{\widetilde}
\newcommand{\dist}{{\rm dist \, }}
\newcommand{\ls}{\lesssim}

\newcommand{\eps}{\varepsilon}

\newcommand{\bproof}{\begin{proof}}
\newcommand{\eproof}{\end{proof}}
\renewcommand{\proofname}{\normalfont{\textbf{Proof}}}
\renewcommand{\qedsymbol}{$\blacksquare$}
\renewcommand{\theenumi}{\roman{enumi}}
\renewcommand{\labelenumi}{{\rm (\theenumi)}}

\newlength{\fixboxwidth}
\newcommand{\fix}[1]{\marginpar{\fbox{\parbox{\fixboxwidth}
{\raggedright\tiny #1}}}}

\newcommand{\be}{\begin{equation}}
\newcommand{\ee}{\end{equation}}
\newcommand{\beq}{\begin{eqnarray}}
\newcommand{\beqq}{\begin{eqnarray*}}
\newcommand{\eeq}{\end{eqnarray}}
\newcommand{\eeqq}{\end{eqnarray*}}

\def\Int{\mbox{Int}}

\def\Id{\mbox{Id}}

\begin{document}
\title{Sampling on energy-norm based sparse grids for the optimal recovery of Sobolev type functions in $H^\gamma$}

\author{Glenn Byrenheid$^a$, Dinh D\~ung$^b$\footnote{Corresponding author. Email: dinhzung@gmail.com}, Winfried Sickel$^c$, Tino Ullrich$^a$ \\\\
$^a$Hausdorff-Center for Mathematics, 53115 Bonn, Germany\\
$^b$Vietnam National University, Hanoi, Information Technology Institute \\
144, Xuan Thuy, Hanoi, Vietnam\\
$^c$ Friedrich-Schiller-University Jena, Ernst-Abbe-Platz 2, 07737 Jena, Germany
}

\date{May 27, 2014}

\maketitle

\begin{abstract} 
We investigate the rate of convergence of linear sampling numbers of the embedding
$H^{\alpha,\beta} (\tor^d) \hookrightarrow H^\gamma (\tor^d)$. Here $\alpha$ governs the mixed smoothness and
$\beta$ the isotropic smoothness in the space $H^{\alpha,\beta}(\tor^d)$ of hybrid smoothness, whereas $H^{\gamma}(\tor^d)$ denotes the isotropic Sobolev space. If $\gamma>\beta$ we obtain sharp polynomial decay rates for the first embedding realized by sampling operators based on ``energy-norm based sparse grids''
for the classical trigonometric interpolation. This complements earlier work by Griebel, Knapek and D\~ung, Ullrich, where general linear approximations have been
considered. In addition, we study the embedding $H^\alpha_{\mix} (\tor^d) \hookrightarrow H^{\gamma}_{\mix}(\tor^d)$ and achieve optimality for Smolyak's algorithm applied to the classical trigonometric interpolation. This can be applied to investigate the sampling numbers for the embedding $H^\alpha_{\mix} (\tor^d) \hookrightarrow L_q(\tor^d)$ for $2<q\leq \infty$ where again Smolyak's algorithm yields the optimal order. The precise decay rates for the sampling numbers in the mentioned situations always coincide with those for the approximation numbers, except probably in the limiting situation $\beta = \gamma$ (including the embedding into $L_2(\tor^d)$). The best what we could prove there is a (probably) non-sharp results with a logarithmic gap between lower and upper bound.
\end{abstract}



\section{Introduction}


The efficient approximation of multivariate functions is a crucial task for the numerical treatment
of several real-world problems. Typically the computation time of approximating algorithms grows
dramatically with the number of variables $d$. Therefore, one is interested in reasonable model assumptions 
and corresponding efficient algorithms. In fact, a large class of solutions of the electronic Schr\"odinger equation 
in quantum chemistry does not only belong to a Sobolev spaces with mixed regularity, one also knows 
additional information in terms of isotropic smoothness properties, see Yserentant's recent lecture notes \cite{Ys10} and 
the references therein. This type of regularity is precisely expressed by the spaces $\hab$, defined in 
Section \ref{spaces} below. Here, the parameter $\alpha$ reflects the smoothness in the dominating mixed sense and 
the parameter $\beta$ reflects the smoothness in the  isotropic sense. We aim at approximating such functions in an 
energy-type norm, i.e.,  we measure the approximation error in an isotropic Sobolev space $H^{\gamma}(\tor^d)$. This is motivated
by the use of Galerkin methods for the $H^1(\tor^d)$-approximation of the solution of general elliptic variational problems see, e.g., \cite{BG99,BG04,GN00,Gr,GN09,SST08}.
The present paper can be seen as a continuation of \cite{DiUl13}, where finite-rank approximations in the sense of approximation numbers were studied. 
The latter are defined as 
$$
  a_m(T:X\to Y):= \inf\limits_{\substack{A:X\to Y\\ \rank A \leq m}}
\sup_{\|f\|_X\leq 1}\, \| Tf-Af\|_Y\quad,\quad m\in \N\,,
$$
where $X,Y$ are Banach spaces and $T\in \mathcal{L}(X,Y)$, where $\mathcal{L}(X,Y)$  denotes the space of all bounded linear operators $T:X\to Y$. In contrast to that, we restrict the class of admissible algorithms even further in this paper and deal with the problem of the optimal recovery
of $H^{\alpha,\beta}$-functions from only a finite number of function values, where the optimality in the worst-case setting is commonly measured in 
terms of linear sampling numbers
$$
  g_m(T:X\to Y):=\inf_{(x_j)_{j=1}^m\subset\tor^d}\inf_{(\psi_j)_{j=1}^m\subset Y}
\sup_{\|f\|_X\leq 1}\, \Big\|\, Tf-\sum_{j=1}^mf(x_j)\psi_j(\cdot)\, \Big\|_Y\quad,\quad m\in \N\,.
$$
Here, $X \subset C(\tor^d)$ denotes a Banach space of functions on $\tor^d$ and $T\in \mathcal{L}(X,Y)$. The inclusion of $X$ in $C(\tor^d)$
is necessary to give a meaning to function evaluations at single points $x_j \in \tor^d$. 

We will mainly focus 
on the situation $X = H^{\alpha,\beta}(\tor^d)$ and $Y = H^{\gamma}(\tor^d)$. The condition $\alpha>\gamma-\beta$ ensures a compact embedding 
\be\label{f07}
    I_1:H^{\alpha,\beta}(\tor^d) \to H^{\gamma}(\tor^d)
\ee
such that we can ask for the asymptotic decay of the sampling numbers 
$$g_m(I_1:H^{\alpha,\beta}(\tor^d)\to H^{\gamma}(\tor^d))$$ in $m$. By investing more isotropic smoothness $\gamma\geq 0$ in the 
target space $H^{\gamma}(\tor^d)$ than $\beta \in \re$ in the source space $H^{\alpha,\beta}$ we encounter two surprising effects for the sampling numbers 
$g_m(I_1)$ if $\gamma>\beta$. The main result of the present paper is the following asymptotic order 
\be
    g_m(I_1)\asymp a_m(I_1) \asymp m^{-(\alpha+\beta-\gamma)}\quad,\quad m\in \N\,,
\ee
which shows, on the one hand, the asymptotic equivalence to the approximation numbers and, on the other hand, the  purely polynomial decay rate, i.e.,  no logarithmic perturbation. In the case $\beta=0$ sampling numbers for these kind of embeddings were also studied in \cite{GH14}. The current paper can be considered as a partial periodic counterpart of the recent papers \cite{Di11, Di13} where the author has investigated the nonperiodic situation, namely sampling recovery in $L_q$-norms as well as corresponding isotropic Sobolev norms of functions on $[0,1]^d$ from Besov spaces $B^{\alpha,\beta}_{p,\theta}$ with hybrid smoothness of mixed smoothness $\alpha$ and isotropic smoothness $\beta$. The asymptotic behavior of the approximation numbers $a_m(I_1:H^{\alpha,\beta}(\tor^d)\to H^{\gamma}(\tor^d))$ (including the dependence of all constants on $d$) has been completely determined in \cite{DiUl13}, see the Appendix in this paper for a listing of all relevant results. The present paper is intended as a 
partial extension of the latter reference to the sampling recovery problem. The general observation is the fact that there is no difference in the asymptotic behavior between sampling and general approximation if we impose certain smoothness conditions on the target spaces $Y$. That is $\gamma>\beta$ if $Y=H^{\gamma}(\tor^d)$ and $\gamma >0$ if $Y =  H^{\gamma}_{\mix}(\tor^d)$. 

It turned out, that the critical cases are $\gamma = \beta \geq 0$. We were not able to give the precise decay rate of 
\be\label{f012}
g_m(I_2:H^{\alpha,\beta}(\tor^d)\to H^{\beta}(\tor^d))
\ee
although we are dealing with a Hilbert space setting and additional smoothness in the target space. However, the 
following statement is true if $\alpha>1/2$. We have
$$
    m^{-\alpha}(\log m)^{(d-1)\alpha} \asymp a_m(I_2) \leq g_m(I_2) \lesssim m^{-\alpha}(\log m)^{(d-1)(\alpha+1/2)}\quad,\quad 2\leq m\in \N\,.
$$

Note, that if $\gamma = \beta = 0$ this includes the classical problem of finding the correct asymptotic behavior of the 
sampling numbers for the embedding
\be\label{f02}
    I_3:H^{\alpha}_{\mix}(\tor^d) \to L_2(\tor^d)\,,
\ee
where $H^{\alpha}_{\mix}(\tor^d)$ denotes the Sobolev space of dominating mixed fractional order $\alpha>1/2$.
Originally brought up by Temlyakov \cite{T85} in 1985, this problem attracted much attention in multivariate approximation theory, see 
D\~ung \cite{Di90,Di91,Di92}, Temlyakov \cite{T85,T93,T93b} and the references therein, Sickel \cite{si02,si06}, and Sickel, Ullrich 
\cite{SU1}-\cite{SU2}. 
Temlyakov himself proved for $\alpha>1/2$ and $2\leq m\in \N$ the estimate
\be\label{f01}
 m^{-\alpha} \, (\log m)^{\alpha(d-1)} \asymp a_m(I_3) \leq g_m(I_3) \lesssim m^{-\alpha} \, (\log m)^{(d-1)(\alpha +1)}\,,
\ee
which was later improved by Sickel, Ullrich \cite{SU1} - \cite{SU2}, D\~ung \cite{Di11}, and Triebel \cite{Tr10} to 
\be\label{f05}
  g_m(I_3:\hmixed \to L_2 (\tor^d)) \lesssim m^{-\alpha} \, (\log m)^{(d-1)(\alpha +1/2)}\quad,\quad 2\leq m\in \N\,.
\ee
The estimate for the approximation numbers in \eqref{f01} can be found in \cite[Theorem~III.4.4]{T93b}. What concerns
the exact $d$-dependence we refer to D\~ung, Ullrich \cite[Theorem\ 4.10]{DiUl13} and the recent contribution K\"uhn, Sickel, Ullrich \cite{KSU2}.
There still remains a logarithmic gap of order $(\log m)^{(d-1)/2}$ between the given upper and lower bounds for the sampling numbers. 
It is a general open problem whether sampling operators 
can be as good as general linear operators in this particular situation. Let us refer to Hinrichs, Novak, Vyb\'iral \cite{H} and Novak, Wo{\'z}niakowski \cite{NoWo11} for relations between approximation and sampling numbers in an general context.
In this paper, we did neither close the gap in \eqref{f01} nor shorten it further. However, we were able to recover these results within our new simplified framework in Subsection \ref{beta00}. 

Surprisingly, the situation becomes much more easy, when we replace in \eqref{f02} the target space $L_2(\tor^d)$ by a Lebesgue space $L_q(\tor^d)$ with $q>2$. 
In fact, we observed for 
the embedding 
\be\label{f011}
  I_4:H^{\alpha}_{\mix}(\tor^d) \to L_q(\tor^d)
\ee
with $\alpha>1/2$ the sharp two-sided estimates 
\be\label{f03}
    g_m(I_4) \asymp a_m(I_4)\asymp 
    \left\{\begin{array}{rcl}
        m^{-(\alpha-1/2+1/q)}(\log m)^{(d-1)(\alpha-1/2+1/q)}&:&2<q<\infty\,,\\
        m^{-(\alpha-1/2)}(\log m)^{\alpha(d-1)}&:&q=\infty\,,
    \end{array}\right.
\ee
for $2\leq m \in \N$\,. The first result of type \eqref{f03} was obtained in \cite{Di90,Di91} for the sampling numbers
$g_m(I: B^{\alpha}_{p,\infty}(\tor^d) \to L_q(\tor^d))$ with $1<p<q\le2$,
  the case $q=\infty$ of  \eqref{f03} was observed by Temlyakov \cite{T93},  we refer to D\~ung \cite{Di11} for 
nonperiodic results of type \eqref{f03}.
 Our method allowed for a significant extension of these results with a shorter proof. 
As a vehicle for $2<q<\infty$ we also took a look to the embedding 
\be\label{f010}
    I_5: H^{\alpha}_{\mix}(\tor^d) \to H^{\gamma}_{\mix}(\tor^d)
\ee
with $\alpha>\max\{\gamma,1/2\}$ and observed 
\be\label{f04}
   g_m(I_5) \asymp a_m(I_5)\asymp m^{-(\alpha-\gamma)}(\log m)^{(d-1)(\alpha-\gamma)}\quad,\quad 2\leq m\in \N\,.
\ee
Let us finally mention that the optimal sampling numbers in \eqref{f03} and \eqref{f04} are realized by the well-known Smolyak algorithm. In other words we presented examples where the Smolyak sampling operator yields optimality. It is also used for the upper bound in \eqref{f05}, but so far not clear whether it is the optimal choice. 

All our proofs are constructive. We explicitly construct sequences of sampling operators that yield the optimal approximation order. Let us briefly describe the framework. The sampling operators will be appropriate sums of tensor products of the classical univariate trigonometric interpolation with respect to the equidistant grid 
$$t^m_{\ell}:=\frac{2\pi \ell}{2m+1}, \qquad \ell =0,1, \ldots \, , 2m\, , $$
given by 
\be\label{f015} 
  I_m f(t) := \frac{1}{2m+1}\, \sum_{\ell=0}^{2m}\, f(t^m_{\ell})\, D_m(t-t^m_{\ell})\,,
\ee
where 
$$ D_m(t):= \sum_{|k|\leq m}e^{ikt} = \frac{\sin((m+1/2)t)}{\sin(t/2)}\, , \qquad t \in \re\,.$$
It is well-known that $I_m f \xrightarrow[m\to \infty]{} f$ in $L_2(\tor)$ for every $f\in H^s(\tor)$ with $s>1/2$\,. Due to telescoping series 
argument we may also write 
$$
      f = I_1 f + \sum\limits_{k=1}^{\infty} (I_{2^k}-I_{2^{k-1}})f.
$$
Therefore, we put for $m\in \N_0$
$$\eta_m :=\begin{cases}
I_{2^m}-I_{2^{m-1}} &  \mbox{if} \quad m>0\, , \\
I_1 &   \mbox{if} \quad m=0\,.
\end{cases}$$
The special structure of the $\eta_m$ immediately admits the following tensorization
\be\label{f09}
   q_k:=\eta_{k_1}\otimes\ldots\otimes\eta_{k_d}\quad,\quad k\in \N_0^d\,.
\ee
Finally, for a given finite $\Delta\subset\N^d_0$ we define the general sampling operator $Q_{\Delta}$ as
\be\label{f06} 
  Q_{\Delta} := \sum_{k\in \Delta}q_k.
\ee
Our degree of freedom will be the set $\Delta$. We will choose $\Delta$ according to the different situations we are dealing with. 
That means in particular that
$\Delta$ may depend on the parameters of the function classes of interest. The most interesting case is represented by the index set 
\be\label{f08}
    \Delta(\xi) = \Delta(\alpha,\beta,\gamma;\xi):=\{k\in \N_0^d:\alpha|k|_1-(\gamma-\beta)|k|_{\infty} \leq \xi\}\quad,\quad \xi>0\,,
\ee
\begin{figure}[H]
\centering
\begin{tikzpicture}[scale=0.16]
    \begin{scope}[thick,font=\scriptsize]
    \draw [->] (0,0) -- (22,0) node [below] {$k_1$};
    \draw [->] (0,0) -- (0,22) node [above] {$k_2$};

   
    \end{scope}
\node at (0,0) {$\bullet$};
\node at (0,1) {$\bullet$};
\node at (0,2) {$\bullet$};
\node at (0,3) {$\bullet$};
\node at (0,4) {$\bullet$};
\node at (0,5) {$\bullet$};
\node at (0,6) {$\bullet$};
\node at (0,7) {$\bullet$};
\node at (0,8) {$\bullet$};
\node at (0,9) {$\bullet$};
\node at (0,10) {$\bullet$};
\node at (0,11) {$\bullet$};
\node at (0,12) {$\bullet$};
\node at (0,13) {$\bullet$};
\node at (0,14) {$\bullet$};
\node at (0,15) {$\bullet$};
\node at (0,16) {$\bullet$};
\node at (0,17) {$\bullet$};
\node at (0,18) {$\bullet$};
\node at (0,19) {$\bullet$};
\node at (0,20) {$\bullet$};
\node at (1,0) {$\bullet$};
\node at (1,1) {$\bullet$};
\node at (1,2) {$\bullet$};
\node at (1,3) {$\bullet$};
\node at (1,4) {$\bullet$};
\node at (1,5) {$\bullet$};
\node at (1,6) {$\bullet$};
\node at (1,7) {$\bullet$};
\node at (1,8) {$\bullet$};
\node at (1,9) {$\bullet$};
\node at (1,10) {$\bullet$};
\node at (1,11) {$\bullet$};
\node at (1,12) {$\bullet$};
\node at (1,13) {$\bullet$};
\node at (1,14) {$\bullet$};
\node at (1,15) {$\bullet$};
\node at (1,16) {$\bullet$};
\node at (1,17) {$\bullet$};
\node at (1,18) {$\bullet$};
\node at (2,0) {$\bullet$};
\node at (2,1) {$\bullet$};
\node at (2,2) {$\bullet$};
\node at (2,3) {$\bullet$};
\node at (2,4) {$\bullet$};
\node at (2,5) {$\bullet$};
\node at (2,6) {$\bullet$};
\node at (2,7) {$\bullet$};
\node at (2,8) {$\bullet$};
\node at (2,9) {$\bullet$};
\node at (2,10) {$\bullet$};
\node at (2,11) {$\bullet$};
\node at (2,12) {$\bullet$};
\node at (2,13) {$\bullet$};
\node at (2,14) {$\bullet$};
\node at (2,15) {$\bullet$};
\node at (2,16) {$\bullet$};
\node at (3,0) {$\bullet$};
\node at (3,1) {$\bullet$};
\node at (3,2) {$\bullet$};
\node at (3,3) {$\bullet$};
\node at (3,4) {$\bullet$};
\node at (3,5) {$\bullet$};
\node at (3,6) {$\bullet$};
\node at (3,7) {$\bullet$};
\node at (3,8) {$\bullet$};
\node at (3,9) {$\bullet$};
\node at (3,10) {$\bullet$};
\node at (3,11) {$\bullet$};
\node at (3,12) {$\bullet$};
\node at (3,13) {$\bullet$};
\node at (3,14) {$\bullet$};
\node at (4,0) {$\bullet$};
\node at (4,1) {$\bullet$};
\node at (4,2) {$\bullet$};
\node at (4,3) {$\bullet$};
\node at (4,4) {$\bullet$};
\node at (4,5) {$\bullet$};
\node at (4,6) {$\bullet$};
\node at (4,7) {$\bullet$};
\node at (4,8) {$\bullet$};
\node at (4,9) {$\bullet$};
\node at (4,10) {$\bullet$};
\node at (4,11) {$\bullet$};
\node at (4,12) {$\bullet$};
\node at (5,0) {$\bullet$};
\node at (5,1) {$\bullet$};
\node at (5,2) {$\bullet$};
\node at (5,3) {$\bullet$};
\node at (5,4) {$\bullet$};
\node at (5,5) {$\bullet$};
\node at (5,6) {$\bullet$};
\node at (5,7) {$\bullet$};
\node at (5,8) {$\bullet$};
\node at (5,9) {$\bullet$};
\node at (5,10) {$\bullet$};
\node at (6,0) {$\bullet$};
\node at (6,1) {$\bullet$};
\node at (6,2) {$\bullet$};
\node at (6,3) {$\bullet$};
\node at (6,4) {$\bullet$};
\node at (6,5) {$\bullet$};
\node at (6,6) {$\bullet$};
\node at (6,7) {$\bullet$};
\node at (6,8) {$\bullet$};
\node at (7,0) {$\bullet$};
\node at (7,1) {$\bullet$};
\node at (7,2) {$\bullet$};
\node at (7,3) {$\bullet$};
\node at (7,4) {$\bullet$};
\node at (7,4) {$\bullet$};
\node at (7,5) {$\bullet$};
\node at (7,6) {$\bullet$};
\node at (8,0) {$\bullet$};
\node at (8,1) {$\bullet$};
\node at (8,2) {$\bullet$};
\node at (8,3) {$\bullet$};
\node at (8,4) {$\bullet$};
\node at (8,5) {$\bullet$};
\node at (8,6) {$\bullet$};
\node at (9,0) {$\bullet$};
\node at (9,1) {$\bullet$};
\node at (9,2) {$\bullet$};
\node at (9,3) {$\bullet$};
\node at (9,4) {$\bullet$};
\node at (9,5) {$\bullet$};
\node at (10,0) {$\bullet$};
\node at (10,1) {$\bullet$};
\node at (10,2) {$\bullet$};
\node at (10,3) {$\bullet$};
\node at (10,4) {$\bullet$};
\node at (10,5) {$\bullet$};
\node at (11,0) {$\bullet$};
\node at (11,1) {$\bullet$};
\node at (11,2) {$\bullet$};
\node at (11,3) {$\bullet$};
\node at (11,4) {$\bullet$};
\node at (12,0) {$\bullet$};
\node at (12,1) {$\bullet$};
\node at (12,2) {$\bullet$};
\node at (12,3) {$\bullet$};
\node at (12,4) {$\bullet$};
\node at (13,0) {$\bullet$};
\node at (13,1) {$\bullet$};
\node at (13,2) {$\bullet$};
\node at (13,3) {$\bullet$};
\node at (14,0) {$\bullet$};
\node at (14,1) {$\bullet$};
\node at (14,2) {$\bullet$};
\node at (14,3) {$\bullet$};
\node at (15,0) {$\bullet$};
\node at (15,1) {$\bullet$};
\node at (15,2) {$\bullet$};
\node at (16,0) {$\bullet$};
\node at (16,1) {$\bullet$};
\node at (16,2) {$\bullet$};
\node at (17,0) {$\bullet$};
\node at (17,1) {$\bullet$};
\node at (18,0) {$\bullet$};
\node at (18,1) {$\bullet$};
\node at (19,0) {$\bullet$};
\node at (20,0) {$\bullet$};
 
\end{tikzpicture}
\caption{$d=2,\;\alpha=2,\;\beta=0,\;\gamma=1,\;\xi=20$}
\label{fig_Squares}
\end{figure}
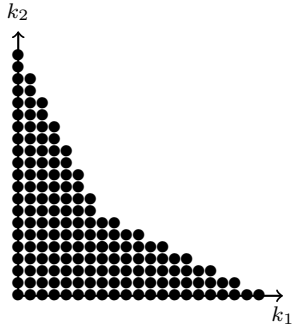
or more exactly, by an $\varepsilon$-modification of it 
given by
\be\label{f082}
\Delta_\varepsilon(\xi) = \Delta(\varepsilon,\alpha,\beta,\gamma;\xi):=\{k\in \N_0^d:\, (\alpha-\varepsilon)\, |k|_1-
(\gamma-\beta-\varepsilon)|k|_{\infty} \leq \xi\}\quad,\quad \xi>0\,,
\ee 
and $\varepsilon>0$ chosen sufficiently small (but not close to zero). These index sets will be used in connection with the embedding \eqref{f07}. 
The set of sampling points used by \eqref{f06} will be called ``energy-norm based sparse grid''. This phrase 
stems from the works of Bungartz, Griebel and Knapek \cite{BG99,BG04,Gr,GN00,GN09} and refers to the special case where the error is measured 
in the ``energy space'' $H^1(\tor^d)$. These authors were the first observing the potential 
of this modification of the classical ``sparse grid''. 
Here we use the phrase ``energy-norm based grids'' in the wider sense of being 
adapted to  the  smoothness parameter $\gamma$ of the target space 
$H^\gamma (\tor^d)$ (with $\alpha$ considered to be fixed). These extensions with respect to 
approximation numbers as well as to sampling numbers have been discussed in  
\cite{Di13} (non-periodic case) and \cite{DiUl13} (periodic case). In particular, \eqref{f08} in case 
$\gamma \neq 1$ goes back to \cite{DiUl13}, and \eqref{f082} in the case $\gamma>0$ to \cite{Di13}.

The second important example is given by the index set
\be\label{f012b}
    \Delta(\xi) = \Delta(\alpha;\xi):=\{k\in \N_0^d:\alpha|k|_1 \leq \xi\}\quad,\quad \xi>0\,,
\ee
\begin{figure}[H]
\centering
\begin{tikzpicture}[scale=0.16]
    \begin{scope}[thick,font=\scriptsize]
    \draw [->] (0,0) -- (22,0) node [below] {$k_1$};
    \draw [->] (0,0) -- (0,22) node [above] {$k_2$};

   
    \end{scope}
\node at (0,0) {$\bullet$};
\node at (0,1) {$\bullet$};
\node at (0,2) {$\bullet$};
\node at (0,3) {$\bullet$};
\node at (0,4) {$\bullet$};
\node at (0,5) {$\bullet$};
\node at (0,6) {$\bullet$};
\node at (0,7) {$\bullet$};
\node at (0,8) {$\bullet$};
\node at (0,9) {$\bullet$};
\node at (0,10) {$\bullet$};
\node at (0,11) {$\bullet$};
\node at (0,12) {$\bullet$};
\node at (0,13) {$\bullet$};
\node at (0,14) {$\bullet$};
\node at (0,15) {$\bullet$};
\node at (0,16) {$\bullet$};
\node at (0,17) {$\bullet$};
\node at (0,18) {$\bullet$};
\node at (0,19) {$\bullet$};
\node at (0,20) {$\bullet$};
\node at (1,0) {$\bullet$};
\node at (1,1) {$\bullet$};
\node at (1,2) {$\bullet$};
\node at (1,3) {$\bullet$};
\node at (1,4) {$\bullet$};
\node at (1,5) {$\bullet$};
\node at (1,6) {$\bullet$};
\node at (1,7) {$\bullet$};
\node at (1,8) {$\bullet$};
\node at (1,9) {$\bullet$};
\node at (1,10) {$\bullet$};
\node at (1,11) {$\bullet$};
\node at (1,12) {$\bullet$};
\node at (1,13) {$\bullet$};
\node at (1,14) {$\bullet$};
\node at (1,15) {$\bullet$};
\node at (1,16) {$\bullet$};
\node at (1,17) {$\bullet$};
\node at (1,18) {$\bullet$};
\node at (1,19) {$\bullet$};
\node at (2,0) {$\bullet$};
\node at (2,1) {$\bullet$};
\node at (2,2) {$\bullet$};
\node at (2,3) {$\bullet$};
\node at (2,4) {$\bullet$};
\node at (2,5) {$\bullet$};
\node at (2,6) {$\bullet$};
\node at (2,7) {$\bullet$};
\node at (2,8) {$\bullet$};
\node at (2,9) {$\bullet$};
\node at (2,10) {$\bullet$};
\node at (2,11) {$\bullet$};
\node at (2,12) {$\bullet$};
\node at (2,13) {$\bullet$};
\node at (2,14) {$\bullet$};
\node at (2,15) {$\bullet$};
\node at (2,16) {$\bullet$};
\node at (2,17) {$\bullet$};
\node at (2,18) {$\bullet$};
\node at (3,0) {$\bullet$};
\node at (3,1) {$\bullet$};
\node at (3,2) {$\bullet$};
\node at (3,3) {$\bullet$};
\node at (3,4) {$\bullet$};
\node at (3,5) {$\bullet$};
\node at (3,6) {$\bullet$};
\node at (3,7) {$\bullet$};
\node at (3,8) {$\bullet$};
\node at (3,9) {$\bullet$};
\node at (3,10) {$\bullet$};
\node at (3,11) {$\bullet$};
\node at (3,12) {$\bullet$};
\node at (3,13) {$\bullet$};
\node at (3,14) {$\bullet$};
\node at (3,15) {$\bullet$};
\node at (3,16) {$\bullet$};
\node at (3,17) {$\bullet$};
\node at (4,0) {$\bullet$};
\node at (4,1) {$\bullet$};
\node at (4,2) {$\bullet$};
\node at (4,3) {$\bullet$};
\node at (4,4) {$\bullet$};
\node at (4,5) {$\bullet$};
\node at (4,6) {$\bullet$};
\node at (4,7) {$\bullet$};
\node at (4,8) {$\bullet$};
\node at (4,9) {$\bullet$};
\node at (4,10) {$\bullet$};
\node at (4,11) {$\bullet$};
\node at (4,12) {$\bullet$};
\node at (4,13) {$\bullet$};
\node at (4,14) {$\bullet$};
\node at (4,15) {$\bullet$};
\node at (4,16) {$\bullet$};
\node at (5,0) {$\bullet$};
\node at (5,1) {$\bullet$};
\node at (5,2) {$\bullet$};
\node at (5,3) {$\bullet$};
\node at (5,4) {$\bullet$};
\node at (5,5) {$\bullet$};
\node at (5,6) {$\bullet$};
\node at (5,7) {$\bullet$};
\node at (5,8) {$\bullet$};
\node at (5,9) {$\bullet$};
\node at (5,10) {$\bullet$};
\node at (5,11) {$\bullet$};
\node at (5,12) {$\bullet$};
\node at (5,13) {$\bullet$};
\node at (5,14) {$\bullet$};
\node at (5,15) {$\bullet$};
\node at (6,0) {$\bullet$};
\node at (6,1) {$\bullet$};
\node at (6,2) {$\bullet$};
\node at (6,3) {$\bullet$};
\node at (6,4) {$\bullet$};
\node at (6,5) {$\bullet$};
\node at (6,6) {$\bullet$};
\node at (6,7) {$\bullet$};
\node at (6,8) {$\bullet$};
\node at (6,9) {$\bullet$};
\node at (6,10) {$\bullet$};
\node at (6,11) {$\bullet$};
\node at (6,12) {$\bullet$};
\node at (6,13) {$\bullet$};
\node at (6,14) {$\bullet$};
\node at (7,0) {$\bullet$};
\node at (7,1) {$\bullet$};
\node at (7,2) {$\bullet$};
\node at (7,3) {$\bullet$};
\node at (7,4) {$\bullet$};
\node at (7,5) {$\bullet$};
\node at (7,6) {$\bullet$};
\node at (7,7) {$\bullet$};
\node at (7,8) {$\bullet$};
\node at (7,9) {$\bullet$};
\node at (7,10) {$\bullet$};
\node at (7,11) {$\bullet$};
\node at (7,12) {$\bullet$};
\node at (7,13) {$\bullet$};
\node at (8,0) {$\bullet$};
\node at (8,1) {$\bullet$};
\node at (8,2) {$\bullet$};
\node at (8,3) {$\bullet$};
\node at (8,4) {$\bullet$};
\node at (8,5) {$\bullet$};
\node at (8,6) {$\bullet$};
\node at (8,7) {$\bullet$};
\node at (8,8) {$\bullet$};
\node at (8,9) {$\bullet$};
\node at (8,10) {$\bullet$};
\node at (8,11) {$\bullet$};
\node at (8,12) {$\bullet$};
\node at (9,0) {$\bullet$};
\node at (9,1) {$\bullet$};
\node at (9,2) {$\bullet$};
\node at (9,3) {$\bullet$};
\node at (9,4) {$\bullet$};
\node at (9,5) {$\bullet$};
\node at (9,6) {$\bullet$};
\node at (9,7) {$\bullet$};
\node at (9,8) {$\bullet$};
\node at (9,9) {$\bullet$};
\node at (9,10) {$\bullet$};
\node at (9,11) {$\bullet$};
\node at (10,0) {$\bullet$};
\node at (10,1) {$\bullet$};
\node at (10,2) {$\bullet$};
\node at (10,3) {$\bullet$};
\node at (10,4) {$\bullet$};
\node at (10,5) {$\bullet$};
\node at (10,6) {$\bullet$};
\node at (10,7) {$\bullet$};
\node at (10,8) {$\bullet$};
\node at (10,9) {$\bullet$};
\node at (10,10) {$\bullet$};
\node at (11,0) {$\bullet$};
\node at (11,1) {$\bullet$};
\node at (11,2) {$\bullet$};
\node at (11,3) {$\bullet$};
\node at (11,4) {$\bullet$};
\node at (11,5) {$\bullet$};
\node at (11,6) {$\bullet$};
\node at (11,7) {$\bullet$};
\node at (11,8) {$\bullet$};
\node at (11,9) {$\bullet$};
\node at (12,0) {$\bullet$};
\node at (12,1) {$\bullet$};
\node at (12,2) {$\bullet$};
\node at (12,3) {$\bullet$};
\node at (12,4) {$\bullet$};
\node at (12,5) {$\bullet$};
\node at (12,6) {$\bullet$};
\node at (12,7) {$\bullet$};
\node at (12,8) {$\bullet$};
\node at (13,0) {$\bullet$};
\node at (13,1) {$\bullet$};
\node at (13,2) {$\bullet$};
\node at (13,3) {$\bullet$};
\node at (13,4) {$\bullet$};
\node at (13,5) {$\bullet$};
\node at (13,6) {$\bullet$};
\node at (13,7) {$\bullet$};
\node at (14,0) {$\bullet$};
\node at (14,1) {$\bullet$};
\node at (14,2) {$\bullet$};
\node at (14,3) {$\bullet$};
\node at (14,4) {$\bullet$};
\node at (14,5) {$\bullet$};
\node at (14,6) {$\bullet$};
\node at (15,0) {$\bullet$};
\node at (15,1) {$\bullet$};
\node at (15,2) {$\bullet$};
\node at (15,3) {$\bullet$};
\node at (15,4) {$\bullet$};
\node at (15,5) {$\bullet$};
\node at (16,0) {$\bullet$};
\node at (16,1) {$\bullet$};
\node at (16,2) {$\bullet$};
\node at (16,3) {$\bullet$};
\node at (16,4) {$\bullet$};
\node at (17,0) {$\bullet$};
\node at (17,1) {$\bullet$};
\node at (17,2) {$\bullet$};
\node at (17,3) {$\bullet$};
\node at (18,0) {$\bullet$};
\node at (18,1) {$\bullet$};
\node at (18,2) {$\bullet$};
\node at (19,0) {$\bullet$};
\node at (19,1) {$\bullet$};
\node at (20,0) {$\bullet$};
 
\end{tikzpicture}
\caption{$d=2,\;\alpha=1,\;\xi=20$}
\label{fig_Squares}
\end{figure}
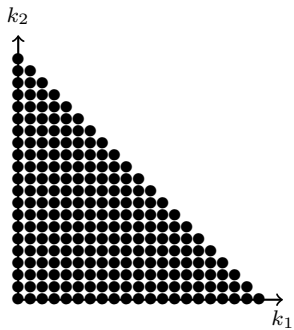
and represents the classical Smolyak algorithm, originally introduced in \cite{Sm}. Although this set represents a special case of \eqref{f08} it has a completely different geometry and leads to structurally different results. The sampling points used by the associated $Q_{\Delta}$ is commonly called ``sparse grid''. 
Putting $\xi = \alpha m$ in \eqref{f012b} it is well-known, see \cite{WW} and \cite{SU,SU1}, that the operator $Q_{\Delta(\xi)}$ 
samples the function $f$ on the grid
\beq
&& \hspace*{-2.7cm} \cg (m):=\Big\{\Big(
\frac{2\pi \ell_1}{2^{j_1+1}+1}, \ldots , \frac{2\pi \ell_d}{2^{j_d+1}+1}\Big): 
\nonumber
\\
&&
\hspace*{2cm} 0 \le \ell_i \leq 2^{j_i}, i=1, \ldots\, , d, \quad m-d+1\le |j|_1 \le m\Big\}\,.
\eeq
It turned out that the previously defined framework fits very well to the function space setting described above. In Lemma \ref{basic1} below we give the 
Littlewood-Paley decomposition of $H^{\alpha,\beta}(\tor^d)$, i.e., 
\[
\hab=\Big\{f\in L_2(\tor^d):\;\fhab^2:=\sum_{k\in\N^d_0}2^{2(\alpha\kone+\beta \kinf)}\deltak^2<\infty\Big\}\,.
\]
As usual, $\delta_k(f)$, $k\in \N_0^d$, represents that part of the Fourier series of $f$ supported in a dyadic block 
\be\label{f014}
    \cp_{k}:= P_{k_1}\times\cdots \times P_{k_d}\,,
\ee
where $P_{j}:=\{\ell\in\mathbb{Z}:2^{j-1}\le |\ell|< 2^{j}\}$ and $P_0 = \{0\}$\,. In fact, looking at the approximation scheme in \eqref{f06} it would be desirable to have an equivalent norm where we replace $\delta_k(f)$ by $q_k(f)$ from \eqref{f09}. Under additional restrictions on the paramaters (one has to at least ensure an embedding in $C(\tor^d)$) this is indeed possible as Theorem \ref{satz:equivnorm} below shows. This gives us convenient characterizations of the function spaces of interest in terms of the sampling operators we are going to analyze. 

The paper is organized as follows.
In Section \ref{spaces} we define and discuss the spaces $\hmixed$ and $\hab$.
Section \ref{sample1} is used to establish our main tool in all proofs involving sampling numbers, the so-called ``sampling representation'', 
see Theorem\ \ref{satz:equivnorm} below. The next Section \ref{sample2} deals in a constructive way with estimates from above for the sampling numbers
of the embedding \eqref{f07} by evaluating the error norm $\|I-Q_{\Delta}\|$ with the corresponding $\Delta$ from \eqref{f082}\,.
With the limiting cases \eqref{f012} leading to the classical Smolyak algorithm we deal in Section \ref{smolyak}. Here we also consider the embeddings \eqref{f010} and \eqref{f011}. In Section \ref{samp} we transfer our approximation results into the notion of sampling numbers and compare them to existing estimates 
for the approximation numbers. The relevant estimates are collected in the appendix.  

{\bf Notation.} As usual, $\N$ denotes the natural numbers, $\N_0$ the non-negative integers,
$\zz$ the integers and
$\re$ the real numbers. With $\tor$ we denote the torus represented by the interval $[0,2\pi]$.
The letter $d$ is always reserved for the dimension in $\Z$, $\R$, $\N^d$, and $\T$.
For $0<p\leq \infty$ and $x\in \R$ we denote $|x|_p = (\sum_{i=1}^d |x_i|^p)^{1/p}$ with the
usual modification for $p=\infty$. We write $e_j$, $j=1,...,d$, for the respective canonical unit vector and 
$\bar{1}:=\sum_{j=1}^d e_j$ in $\R$.  
If $X$ and $Y$ are two Banach spaces, the norm of an operator
$A:~X \to Y$ will be denoted by $\|A:~X\to Y\|$. The symbol $X \hookrightarrow Y$ indicates that there is a
continuous embedding from $X$ into $Y$. The relation $a_n\lesssim b_n$ means that there is a constant $c>0$ independent of the context relevant parameters 
such that $a_n\le c b_n$ for all $n$ belonging to a certain subset of $\N$, often $\N$ itself. We write $a_n \asymp b_n$ if 
$a_n\lesssim b_n$ and $b_n \lesssim a_n$ holds.


\section{Sobolev-type  spaces}\label{spaces}


In this section we recall the definition of the function spaces under consideration here.
They are all of Sobolev-type.
In a first subsection we consider the periodic Sobolev spaces 
$\hmixed$ of dominating mixed fractional order $\alpha>0$.
In the second subsection the more general classes $\hab$ are discussed.


\subsection{Periodic Sobolev spaces of mixed and isotropic smoothness}\label{spacesa}


All results in this paper are stated for function spaces on the $d$-torus $\T$,
which is represented in the Euclidean space $\R$ by the cube
$\tor^d = [0,2\pi]^d$, where opposite faces are identified.
The space $L_2(\T)$ consists of all (equivalence classes of) measurable functions $f$ on $\T$
such that the norm
$$
    \|f\|_2:=\Big(\int_{\T} |f(x)|^2\,dx\Big)^{1/2}
$$
is finite. All information on a function $f\in L_2(\T)$ is encoded
in the sequence $(c_k(f))_k$ of its Fourier coefficients, given by
\[
c_k (f):= \frac{1}{(2\pi)^d} \, \int_{\T} \, f(x)\, e^{-ikx}\, dx\, , \qquad k\in \Z\, .
\]
Indeed, we have Parseval's identity
\be\label{Pars}
    \|f\|_2^2 = (2\pi)^d \, \sum\limits_{k\in \Z} |c_k(f)|^2
\ee
as well as
\[
f(x) = \sum_{k \in \Z} \, c_k (f)\, e^{ikx}
\]
with convergence in $L_2(\T)$.

The mixed Sobolev space $H^{m}_\mix(\T)$ with smoothness vector $m = (m_1,...,m_d)\in \N^d$ is the collection
of all $f\in L_2 (\T)$ such that all
distributional derivatives $D^\gamma  f$ of order $\gamma = (\gamma_1,...,\gamma_d)$ with
$\gamma_j\le m_j$, $j=1,...,d$, belong to $L_2 (\T)$.
We put
\begin{equation}\label{t-1}
\| \, f \,\|_{H^{m}_{\mix} (\T)}^* := \Big(
\sum_{\substack{0\leq \gamma_j\leq m_j\\ j=1,...,d}}
\|D^\gamma  f\|_2^2\Big)^{1/2} \, .
\end{equation}
One can rewrite this definition in terms of Fourier coefficients.
However, it is more convenient to use an equivalent norm like
\be\label{norm3}
\| \, f \, \|_{H^{m}_\mix (\T)}^{\#} := \, \Big[
 \sum_{k \in \Z} \, |c_k (f)|^2 \prod\limits_{j=1}^d\big(1+|k_j|^2\big)^{m_j} \Big]^{1/2}\, .
\ee
For $m\in \N$ we denote with $H^m(\tor^d)$ the space $H^{m\cdot \bar{1}}(\tor^d)$.
Inspired by  \eqref{norm3} we define Sobolev spaces of dominating 
mixed smoothness of fractional order $\alpha$ as follows.

\begin{defi}\label{varnorms}
Let $\alpha>0$. The periodic Sobolev space $H^\alpha_{\mix}(\T)$  of dominating mixed smoothness $\alpha$ 
is the collection of all $f\in L_2(\T)$ such that
\be\label{norm4}
   \| \, f \,\|_{H^\alpha_{\mix} (\T)}^{\#}  :=
\, \Big[
 \sum_{k \in \Z} \, |c_k (f)|^2 \prod\limits_{j=1}^d\big(1+|k_j|^2\big)^{\alpha} \Big]^{1/2} < \infty \, .
\ee
\end{defi}

\begin{rem}
 \rm There is different notation in the literature.
E.g., Temlyakov and others use $MW^\alpha_2 (\tor^d)$ instead of $H^\alpha_{\mix} (\T)$, whereas Amanov, Lizorkin, 
Nikol'skij, Schmeisser and Triebel prefer to use $S^\alpha_2W(\tor^d)$.
\end{rem}

We also need the (isotropic) Sobolev spaces $H^\gamma (\tor^d)$.

\begin{defi}\label{varnorms3}
Let $\gamma \ge 0$. The periodic Sobolev space $H^\gamma(\T)$  of smoothness $\gamma$ 
is the collection of all $f\in L_2(\T)$ such that
\be\label{norm7}
   \| \, f\,  \|_{H^\gamma (\T)}^{\#}  :=
\, \Big[
 \sum_{k \in \Z} \, |c_k (f)|^2 \big(1+|k|_2^2\big)^{\gamma} \Big]^{1/2} < \infty \, .
\ee
\end{defi}

\begin{rem}
 \rm
It is elementary to check
\[
H^{\alpha d} (\T) \hookrightarrow H^\alpha_\mix (\tor^d) \hookrightarrow H^\alpha (\tor^d)\, .
\]
In addition it is known that $H^\gamma (\T) \hookrightarrow C(\tor^d)$ if and only if 
$H^\gamma (\T) \hookrightarrow L_\infty(\tor^d)$ if and only if $\gamma >d/2$, see \cite{WiTr}.
\end{rem}


\subsection{Hybrid type Sobolev spaces}\label{spacesb}


To define the scale $\hab$ we look for subspaces of $\hmixed$ obtained by adding isotropic smoothness.
To make this more transparent we start again with a situation where smoothness can be described 
exclusively in terms of weak derivatives. It is easy to see that isotropic smoothness of order $n\in \N$ can be achieved by
``intersecting'' mixed smoothness conditions, i.e.,
$$
    H^{n}(\tor^d) = H^{(n,0,...,0)}_{\mix}(\tor^d) \cap H^{(0,n,0,..,0)}_{\mix} \cap... \cap H^{(0,0,...,n)}_{\mix}\,.
$$
Let $m \in \N$ and $n\in \zz$ such that $m+n\geq 0$. We will use the above principle to ``add'' an 
isotropic smoothness of order $n$ to the mixed smoothness of order $m$. The hybrid type Sobolev space $H^{m,n}(\T)$ is the set
$$
    H^{m,n}(\tor^d) = \left\{\begin{array}{rcl}
                                \bigcap\limits_{j=1}^d H^{m \cdot \bar{1}+ne_j}_{\mix}(\tor^d)&:&n\geq 0\,,\\
                                \sum\limits_{j=1}^d H^{m \cdot \bar{1}+ne_j}_{\mix}(\tor^d)&:& n<0\,.
                             \end{array}\right.
$$
A function $f \in L_2(\T)$ belongs to $H^{m,n}(\tor^d)$, if and only if the semi-norm
$$
    |f|'_{H^{m,n}(\tor^d)} = \left\{\begin{array}{rcl}
                               \max_{1\le j \le d}  \|f\|_{H^{m \cdot \bar{1}+ne_j}_{\mix}(\tor^d)}&:
&n\geq 0\,,\\[2.0ex]
                                \min_{1\le j \le d}  \|f\|_{H^{m \cdot \bar{1}+ne_j}_{\mix}(\tor^d)}&:& n<0\,,
                             \end{array}\right.
$$
is finite. The norm of $f$ in $H^{m,n}(\T)$ is defined as 
$\|f\|'_{H^{m,n}(\tor^d)}:= \|f\|_2 + |f|'_{H^{m,n}(\tor^d)}$. Hence, one can verify that 
\be\nonumber
   \| \, f \, \|'_{H^{m, n} (\T)}
\ \asymp \ 
\, \Big[
 \sum_{k \in \Z} \, |c_k (f)|^2 \Big(\prod\limits_{j=1}^d\big(1+|k_j|^2\big)^{m}\Big) (1+|k|_2^2)^n
 \Big]^{1/2}\, .
\ee
This motivates the following definition.
\begin{defi}\label{varnorms2}
Let $\alpha\geq 0$ and $\beta \in \re$ such that $\alpha+\beta\geq 0$. The generalized periodic Sobolev space $H^{\alpha,\beta}(\T)$ is the collection of all $f\in L_2(\T)$ such that
\be\label{norm6}
   \| \, f \, \|_{H^{\alpha, \beta} (\T)}^{\#}  :=
\, \Big[
 \sum_{k \in \Z} \, |c_k (f)|^2 \Big(\prod\limits_{j=1}^d\big(1+|k_j|^2\big)^{\alpha}\Big) (1+|k|_2^2)^\beta
 \Big]^{1/2} < \infty \, .
\ee
\end{defi}

\begin{rem}
 \rm
(i) Obviously we have 
$H^{\alpha, 0}_{\mix} (\T)= H^{\alpha}_{\mix} (\T)$ and $H^{0, \beta}_{\mix} (\T)= H^{\beta}(\T)$, $\beta \ge 0$.
More important for us will be the embedding
\be\label{ws-24}
\hab \hookrightarrow H^{\gamma}(\T) \qquad \mbox{if}\qquad 0 \leq \gamma \le \alpha + \beta\, .
\ee
(ii) Spaces of such a type have been first considered by Griebel and Knapek \cite{GN00}.
Also in the non-periodic context they play a role in the description of the fine regularity properties of certain eigenfunctions 
of Hamilton operators in quantum chemistry, see \cite{Ys10}.
The periodic spaces $H^{\alpha, \beta}_{\mix} (\T)$ also occur in the recent works \cite{DiUl13} and \cite{GH14}.
\end{rem}
A first step towards the sampling representation in Theorem \ref{satz:equivnorm} below will be the following equivalent characterization of Littlewood-Paley type. 
We will work with the dyadic blocks from \eqref{f014} and put for $\ell \in \N_0^d$
\[
\delta_{\ell}(f):=\sum_{k \in \cp_{\ell}} \, c_k(f)\, e^{ikx}.
\]
Hence, for all $f \in L_2 (\tor^d)$ we have the Littlewood-Paley decomposition
\be\label{ws-18}
f = \sum_{\ell \in \N_0^d} \delta_{\ell}(f)\,.
\ee
The following lemma is an elementary  consequence of Definition \ref{varnorms2}.

\begin{lem}\label{basic1}
Let $\alpha \geq 0$ and $\beta\in\re$ such that $\alpha+\beta\geq 0$. \\
{\em (i)} Then 
\[
\hab=\Big\{f\in L_2(\tor^d):\;\|f\|_{H^{\alpha,\beta}(\tor^d)}:=\Big(\sum_{k\in\N^d_0}2^{2(\alpha\kone+\beta \kinf)}\deltak^2\Big)^{1/2}<\infty\Big\}
\]
in the sense of equivalent norms.\\
{\em (ii)} We have
$$
    H^{\alpha,\beta}(\tor^d) = \left\{\begin{array}{rcl}
                                \bigcap\limits_{j=1}^d H^{\alpha \cdot \bar{1}+\beta e_j}_{\mix}(\tor^d)&:&\beta\geq 0\,,\\
                                \sum\limits_{j=1}^d H^{\alpha \cdot \bar{1}+\beta_j}_{\mix}(\tor^d)&:& \beta<0\,.
                             \end{array}\right.
$$
\end{lem}
\noindent We need a few more properties of these spaces. For $\ell\in\N^d_0$ we define the set of trigonometric polynomials
$$\trigpol:= \Big\{\sum_{\substack{|k_i|\le 2^{\ell_i}\\i=1,\cdots,d}} a_k \, e^{ikx}:a_k \in \C\Big\}\, .$$
Of course, $\delta_\ell (f) \in \trigpol$ for all $f \in L_2 (\tor^d)$.

\begin{lem}[Nikol'skij's inequality]\label{lem:bernsteinnikolskij}
Let $0<p\leq q \leq \infty$. Then there is a constant $C = C(p,q)>0$ (independent of $g$ and $\ell$) sucht that
\beqq 
\|g\|_q\leq C 2^{\lone(\frac{1}{p}-\frac{1}{q})}\|g\|_p
\eeqq
holds for every $g\in\trigpol$ and every $\ell\in\N_0^d$.
\end{lem}

\begin{proof}
A proof can be found in  \cite[Theorem~3.3.2]{ST87}.
\end{proof}

To give a meaning to point evaluations of functions it is essential that the spaces under consideration contain only continuous functions. To be more 
precise, they contain equivalence classes of functions having one continuous representative.

\begin{satz}\label{einbettung}
Let $\alpha>0$, $\beta\in \re$ such that $\min\{\alpha+\beta,\alpha+\frac{\beta}{d}\}>\frac{1}{2}$. Then
$$\hab\hookrightarrow C(\tor^d).$$
\end{satz}

\begin{proof}
Applying Lemma \ref{lem:bernsteinnikolskij} yields
\beqq
\sum_{k\in\N_0^d}\|\delta_k(f)\|_{\infty}&=& \sum_{k\in\N_0^d}2^{\alpha \kone+\beta\kinf}2^{-(\alpha\kone+\beta\kinf)}\|\delta_k(f)\|_{\infty}\\
&\lesssim& \sum_{k\in\N_0^d}2^{\alpha \kone+\beta\kinf}2^{-(\alpha\kone+\beta\kinf)}2^{\frac{\kone}{2}}\|\delta_k(f)\|_{2}\, .  
\eeqq
Employing H\"older's inequality we find
\beqq
\sum_{k\in\N_0^d}\|\delta_k(f)\|_{\infty}&\leq& \Big(\sum_{k\in\N_0^d}2^{-2(\alpha\kone+\beta\kinf)}2^{\kone}\Big)^{\frac{1}{2}} \Big(\sum_{k\in\N_0^d}2^{2(\alpha \kone+\beta\kinf)}\|\delta_k(f)\|_{2}\Big)^{\frac{1}{2}}\\
&\leq&\Big(\sum_{k\in\N_0^d}2^{-2(\alpha\kone+\beta\kinf)}2^{\kone}\Big)^{\frac{1}{2}} \|f\|_{\hab}.
\eeqq
Using $\kinf \leq \kone \leq d\kinf$ gives in case
$\beta\geq 0$
\beqq
\sum_{k\in\N_0^d}2^{-2(\alpha\kone+\beta\kinf)}2^{\kone}&\leq& \sum_{k\in\N_0^d}2^{-2(\alpha+\frac{\beta}{d}-\frac{1}{2})\kone}<\infty\, , 
\eeqq
whenever $\alpha+\frac{\beta}{d}>\frac{1}{2}$. For the case $\beta <0$ observe that
\beqq
\sum_{k\in\N_0^d}2^{-2(\alpha\kone+\beta\kinf)}2^{\kone}&\leq& \sum_{k\in\N_0^d}2^{-2(\alpha+\beta-\frac{1}{2})\kone}<\infty
\eeqq
if $\alpha+\beta>\frac{1}{2}$. Since $C(\tor^d)$ is a Banach space, the sum $\sum_{k\in\N_0^d}\delta_k(f)$ belongs to $C(\tor^d)$ due to its absolute 
convergence. Further
$$f=\sum_{k\in\N_0^d}\delta_k(f)$$ holds in $L_2(\tor^d)$. Consequently, the equivalence class $f \in \hab$ has 
a continous representative.
\end{proof}

\begin{rem}\label{optimal} \rm {(i)} With essentially the same proof technique as above the assertion in Theorem \ref{einbettung} can be refined as follows. 
Let $\alpha\geq 0$ and $\beta \in \re$ such that $\alpha+\beta \geq 0$. Then it holds the embedding
$$
    H^{\alpha,\beta}(\tor^d) \hookrightarrow \left\{\begin{array}{rcl}
                                                  H^{\alpha+\beta/d}_{\mix}(\tor^d)&:&\beta\geq 0,\\
                                                  H^{\alpha+\beta}_{\mix}(\tor^d)&:&\beta<0\,.
                                             \end{array}\right.
$$
This embedding immediately implies Theorem \ref{einbettung}\,.\\
{(ii)} The restrictions in Theorem \ref{einbettung} are almost optimal. Indeed, let $g \in H^{\alpha + \beta} (\tor)$, then the function
\[
f(x_1, \ldots \, , x_d):= g(x_1) \, , \qquad x \in \R\, , 
\]
belongs to $\hab$. Hence, from $\hab \hookrightarrow C(\tor^d)$ we derive $H^{\alpha + \beta} (\tor) \hookrightarrow C(\tor)$ which is known to be true
if and only if $\alpha + \beta >1 /2$.
In case $\alpha =0$ we know $\hab = H^\beta (\tor^d)$. Hence, $H^{0,\beta} \hookrightarrow C(\tor^d)$ if and only if $\beta /d >1 /2$. 
\end{rem}

We will need the following Bernstein type inequality.

\begin{lem} \label{lem:bernstein}
Let $\min\{\alpha,\alpha+\beta-\gamma\}>0$ and $\ell\in\N^d_0$. Then
\be \label{eq:mixedbernstein}
\fhab\leq 2^{\alpha\lone+(\beta-\gamma)\linf}\|f\|_{H^{\gamma}}
\ee
holds for all $f\in\trigpol$.
\end{lem}

\begin{proof}
Indeed, for $f\in\trigpol$, we have
\beqq
\fhab^2&=&\sum_{\substack{k_i\leq \ell_i\\i=1,\cdots,d}} 2^{2(\alpha|k|_1+\beta\kinf}\, \deltak^2
\leq  \max_{\substack{k_i\leq \ell_i\\i=1,\cdots,d}} 2^{2(\alpha |k|_1+(\beta-\gamma)\kinf)}\, 
\sum_{\substack{k_i\leq\ell_i\\i=1,\cdots,d}} 2^{2\gamma |k|_{\infty}}\deltak^2\\
&\leq& 2^{2(\alpha \lone+(\beta-\gamma)\linf)}\, \|f\|^2_{H^{\gamma}}.
\eeqq
\end{proof}


\section{Sampling representations}
\label{sample1}


Our main aim in this section consists in deriving a specific Nikol'skij-type representation 
for the spaces $\hab$ in the spirit of Lemma \ref{basic1}. Specific in the sense, that the building blocks in 
the decomposition originate from associated sampling operators of type \eqref{f09}.
First we need some technical lemmas.\\

\begin{lem}\label{lem:monotonicity}
Let $\alpha>0$, $\beta \in \mathbb{R}$, $\min\{\alpha,\alpha+\beta\}>0$ and 
$$\psi(k):=\alpha\kone+\beta\kinf\quad,\quad k\in \N_0^d.$$ 
Then there is an $\eps>0$ such that 
$$\psi(k)\leq \psi(k')-\varepsilon(|k'|_1-\kone)$$
holds for all $k', k \in \N_0^d$ with $k'\geq k$ component-wise.
\end{lem}

\begin{proof}Let $k'\geq k$. This implies
\be\label{ws-20}
\psi(k) = \psi(k')-\alpha|k'-k|_1-\beta(|k'|_{\infty}-|k|_{\infty})
\ee
We need to distinguish two cases.\\ 
Case 1. If $\beta\geq0$ we have as an immediate consequence of \eqref{ws-20}
$$\psi(k)\leq \psi(k')-\alpha|k'-k|_1\,.$$
\\
Case 2. Let $\beta<0$. From \eqref{ws-20} and 
$$
  |k'|_{\infty} - |k|_{\infty} \leq |k'-k|_{\infty} \leq |k'-k|_1
$$
we obtain
$$
   \psi(k) \le  \psi(k')-(\alpha+\beta)|k'-k|_1\,.
$$
\end{proof}

Recall the linear operator $q_k$ has been defined in \eqref{f09}.  Let us settle the following cancellation property. 

\begin{lem}\label{lem:cancell} Let $\ell,k \in \N_0^d$ with $k_n<\ell_n$ for some $n \in \{1,...,d\}$. Let further $f \in T^k$ and $q_\ell$ be the operator defined in \eqref{f09}. Then $q_{\ell}(f) = 0$.
\end{lem}

\begin{proof} Since $f \in \mathcal{T}^k$ we have
$$
    f = \sum_{\substack{|m_j|\le 2^{k_j}\\j=1,\cdots,d}} a_m \, e^{i m x} $$
and 
$$
   q_\ell(f)(x) = \sum\limits_{\substack{|m_j|\le 2^{k_j}\\j=1,\cdots,d}} a_m q_{\ell}(e^{im\cdot})(x) = \sum\limits_{\substack{|m_j|\le 2^{k_j}\\j=1,\cdots,d}} a_m \prod\limits_{j=1}^d \eta_{\ell_j}(e^{im_j\cdot})(x_j)\,.
$$
Due to $2^{\ell_n-1}\geq 2^{k_n}\geq m_n$ we have
$$
    \eta_{\ell_n}(e^{im_n\cdot})(x_n) = (I_{2^{\ell_n}}-I_{2^{\ell_n-1}})(e^{im_n\cdot})(x_n) = 0\,
$$
which implies $q_{\ell}(f) = 0$\,. 
\end{proof}

\noindent Now we are in the position to proof  Nikol'skij's type representation theorems for the spaces $\hab$.

\begin{prop}\label{lem:lowernorm}
Let $\min (\alpha,\alpha + \beta) > 1/2$. Then every function $f \in \hab$ can be represented by the  series 
\be \label{Representation}
f
\ = \
\sum_{k\in\N_0^d}q_k(f)
\ee
converging unconditionally in $\hab$, and satisfying the condition
\be \label{ConvegenceCondition[hab]}
 \sum_{k\in \N^d_0}2^{2(\alpha\kone+\beta\kinf)}\|q_k(f)\|_2^2\leq C \fhab^2 
\ee
with a constant $C=C(\alpha,\beta,d)>0$. 
\end{prop}

\begin{proof}
{\em Step 1}. We first prove \eqref{ConvegenceCondition[hab]} for $f \in \hab$. Let us assume $\beta\neq 0$, otherwise set $\beta=\tbeta=0$. For technical reasons we need to fix $\talpha, \zeta, \tbeta \in \re $ sucht that 
\be\label{requirements}\min\{\talpha-\zeta,\;\talpha-\zeta+\tbeta\}>0,\;\alpha-\talpha>0, \;\tbeta<\beta\mbox{ and }\zeta>\frac{1}{2}\ee
holds. For $\beta>0$ it is easy to find parameters $\talpha, \zeta, \tbeta$ fulfilling \eqref{requirements}. Critical is the case $\beta<0$. Here we choose the parameters in the following way:
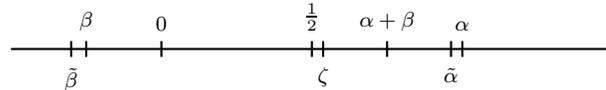
\begin{figure}[H]
\centering
\begin{tikzpicture}
    \begin{scope}[thick,font=\scriptsize]
    \draw (0,0) -- (8,0);

   
    \draw (2,-3pt) -- (2,3pt)   node [above] {$0$};
    \draw (4,-3pt) -- (4,3pt)   node [above] {$\frac{1}{2}$};
    \draw (1,-3pt) -- (1,3pt)   node [above] {$\beta$};
    \draw (6,-3pt) -- (6,3pt)   node [above] {$\alpha$};
    \draw (5,-3pt) -- (5,3pt)   node [above] {$\alpha+\beta$};
    \draw (5.85,3pt) -- (5.85,-3pt)   node [below] {$\talpha$};
	\draw (4.15,3pt) -- (4.15,-3pt)   node [below] {$\zeta$};
	\draw (0.8,3pt) -- (0.8,-3pt)   node [below] {$\tbeta$};
    \end{scope}

\end{tikzpicture}
\label{fig_Squares}
\end{figure}
\noindent The condition $\alpha+\beta>\frac{1}{2}$ implies that there is some 
$\varepsilon>0$ such that $\alpha+\beta-\varepsilon> \frac{1}{2}$ holds. Choose now $\talpha, \tbeta, \zeta \in\re$  s.t.
$\beta-\frac{\varepsilon}{2} < \tbeta < \beta$ and  $\frac{1}{2} < \zeta <\talpha<\alpha$ with
$$ 0 < \alpha-\frac{1}{2}-\frac{\varepsilon}{2}<\talpha-\zeta<\alpha-\frac{1}{2}.$$
Obviously this is possible. It is easy to check that such a choice fulfills the properties in \eqref{requirements}
\beqq
\talpha-\zeta+\tbeta&>& \Big(\alpha-\frac{1}{2}-\frac{\varepsilon}{2})+\Big(\beta-\frac{\varepsilon}{2}\Big)\\
&=&\Big(\alpha+\beta-\varepsilon\Big)-\frac{1}{2}\\
&>& 0.
\eeqq
 We claim that there exists a constant $c$ such that
\be\label{ws-21} 
2^{\talpha\lone+\tbeta\linf}\|q_{\ell}(f)\|_2\le c\,  \Big(
\sum_{\substack{k_i\geq\ell_i\\i=1,\cdots,d}}2^{2(\talpha\kone+\tbeta\kinf)}\|\delta_k(f)
\|^2_2\Big)^{\frac{1}{2}}
\ee
holds for all $\ell \in \N_0^d$. 
Because of  $f=\sum_{k\in\N^d_0}\delta_k(f)$  and linearity of $q_k$ we have
\[
\|q_{\ell}(f)\|_2  =  \Big\|\sum_{k\in\N^d_0}q_{\ell}(\delta_k(f))\Big\|_2\, .
\]
Using $\delta_k (f)\in \trigpolk$, Lemma \ref{lem:cancell}, and the triangle inequality we find
\[
\|q_{\ell}(f)\|_2 = \Big\|\sum_{\substack{k_i\geq\ell_i\\i=1,\cdots,d}}q_{\ell}(\delta_k(f))\Big\|_2
\leq \sum_{\substack{k_i\geq\ell_i\\i=1,\cdots,d}}\|q_{\ell}(\delta_k(f))\|_2.
\]
Using Lemma 5 in \cite{SU} and known results about the approximation power of the $I_m$, see \cite{si01},
we  obtain
\beqq \|q_{\ell}(f)\|_2\lesssim 2^{-\zeta\lone}\sum_{\substack{k_i\geq\ell_i\\i=1,\cdots,d}}\|\delta_k(f)\|_{\hmixedtta}.
\eeqq
Lemma \ref{lem:bernstein} yields
\beqq
\|q_{\ell}(f)\|_2\lesssim  2^{-\zeta\lone}\sum_{\substack{k_i\geq\ell_i\\i=1,\cdots,d}}2^{\zeta\kone}\|\delta_k(f)\|_2.
\eeqq
We proceed by inserting an additional weight and apply H\"older's inequality 
\be \label{eq:1234}
\|q_{\ell}(f)\|_2 \lesssim   2^{-\zeta\lone}\Big(\sum_{\substack{k_i\geq \ell_i\\i=1,\cdots,d}}2^{-2[(\talpha-\zeta)\kone+\tbeta\kinf]}\Big)^{\frac{1}{2}}\Big(\sum_{\substack{k_i\geq \ell_i\\i=1,\cdots,d}}2^{2(\talpha\kone+\tbeta\kinf)}\|\delta_k(f)\|^2_2\Big)^{\frac{1}{2}}.
\ee
Lemma \ref{lem:monotonicity} with $\xi>0$ chosen such that $\min\{\talpha-\zeta,\talpha-\zeta+\tbeta\}\ge \xi$ leads to
\beqq
\sum_{\substack{k_i\geq \ell_i\\i=1,\cdots,d}}2^{-2[(\talpha-\zeta)\kone+\tbeta\kinf]}
&\leq&
2^{-2[(\talpha-\zeta)\lone+\tbeta\linf]}\sum_{\substack{k_i\geq \ell_i\\i=1,\cdots,d}} \, 2^{-2\xi|k-\ell|_1}\\
&\lesssim&2^{-2[(\talpha-\zeta)\lone+\tbeta\linf]}.
\eeqq
Inserting this into \eqref{eq:1234} proves \eqref{ws-21}.

Taking squares and summing up with respect to $\ell$ in \eqref{ws-21} we get
\beqq 
\sum_{\ell\in\N^d_0}2^{2(\alpha\lone+\beta\linf)}\|q_{\ell}(f)\|^2_2 &\lesssim&  \sum_{\ell\in\N^d_0} 2^{ 2[(\alpha-\talpha)\lone+(\beta-\tbeta)\linf] }\sum_{\substack{k_i\geq \ell_i \\ i=1,\cdots,d}}2^{2(\talpha\kone+\tbeta\kinf)}\deltak^2.
\eeqq
Next, interchanging the order of summation  yields
\beqq \sum_{\ell\in\N^d_0}2^{2(\alpha\kone+\beta\kinf)}\|q_{\ell}(f)\|^2_2&\lesssim& \sum_{k\in\N^d_0} 2^{2(\talpha\kone+\tbeta\kinf)}\deltak^2\sum_{\substack{\ell_i\leq k_i\\i=1,\cdots,d}}2^{2((\alpha-\talpha)\lone+(\beta-\tbeta)\linf}.\\
\eeqq
One more time we apply Lemma \ref{lem:monotonicity}, this time with $0 < \xi \le \alpha-\talpha$, which results in
\beqq 
 \sum_{\ell\in\N^d_0} && \hspace{-0.6cm} 2^{2(\alpha\kone+\beta\kinf)} \, \|q_{\ell}(f)\|^2_2
\\
&\leq&
\sum_{k\in\N^d_0} 2^{2(\talpha\kone+\tbeta\kinf)}\, \deltak^2 \, \,  2^{2((\alpha-\talpha)\kone+(\beta-\tbeta)\kinf)} 
\sum_{\substack{\ell_i\leq k_i\\i=1,\cdots,d}} 2^{-2\xi|k-\ell|_1}
\\
&\lesssim& \sum_{k\in\N^d_0}\,  2^{2(\alpha\kone+\beta\kinf)}\, \deltak^2.
\eeqq
This proves \eqref{ConvegenceCondition[hab]}. \\\\
{\em Step 2.} Let $f \in \hab$. We will show that $f$ can be represented by the  series \eqref{Representation}
converging in the norm of $\hab$. Applying Lemma \ref{lem:bernstein}, H\"older's inequality and 
\eqref{ConvegenceCondition[hab]} yields
\beq
\sum_{k\in\N_0^d}\|q_k(f)\|_{\hab} &\leq& \sum_{k\in\N_0^d}2^{\alpha\kone+\beta\kinf}\|q_k(f)\|_2\nonumber\\
&\leq&C\|f\|_{\hab}<\infty.\label{eq:upqkf}
\eeq
Hence $\sum_{k\in\N_0^d}\|q_k(f)\|_{\hab}<\infty$ and therefore $\sum_{k\in\N_0^d} q_k(f)$ converges unconditionally in $\hab$ if $\min\{\alpha,\alpha+\beta\}>\frac{1}{2}$. We denote the limit as $F:=\sum_{k\in\N_0^d}q_k(f)$. By the definition of the norm in $\hab$
$$\fhab^2=\sum_{\ell\in\N_0^d}2^{2(\alpha\lone+\beta\linf)}\deltal^2$$ 
we see that the trigonometric polynomials are dense in $\hab$. Let now $t$ be a trigonometric polynomial. We consider
$\|F-t\|_{\hab}$. Clearly, $t=\sum_{k\in\N_0^d}q_k(t)$ and, by definition, $F=\sum_{k\in\N_0^d}q_k(f)$ implying
\be \label{eq:ft} F-t=\sum_{k\in\N_0^d}q_k(f-t)\ee with convergence in $\hab$ for every trigonometric polynomial $t$. Now, for every trigonometric polynomial $t$ we have
\be \label{eq:ftnorm} \|F-f\|_{\hab}\leq \|F-t\|_{\hab}+\|t-f\|_{\hab}.\ee
By \eqref{eq:upqkf} and \eqref{eq:ft} we get
$$\|F-t\|_{\hab}\leq C\|f-t\|_{\hab}.$$
Putting this into \ref{eq:ftnorm} yields
$$\|F-f\|_{\hab}\leq (C+1)\|f-t\|_{\hab}.$$
Choosing $t$ close enough to $f$ gives
$$\|F-f\|_{\hab}<\varepsilon$$
for all $\varepsilon>0$ and hence $\|F-f\|_{\hab}=0$ which is
$$f=\sum_{k\in\N_0^d}q_k(f)$$
in $\hab$.

\end{proof}

\begin{prop}\label{lem:uppernorm}
Let $\beta\in\re$, $\min\{\alpha,\alpha+\beta\}>0$ and $(f_k)_{k\in\N_0^d}$ a sequence with $f_k\in\trigpolk$ satisfying
$$\sum_{k\in \N_0^d}2^{2(\alpha\kone+\beta\kinf)}\|f_k\|_2^2<\infty.$$
Assume that the series $\sum_{k\in\N_0^d}f_k$ converges in $L_2 (\tor^d)$ to a function $f$. Then $f \in \hab$, and moreover, there is a constant $C=C(\alpha,\beta,d)>0$ such that 
\be 
\fhab^2 \leq C \sum_{k\in\N_0^d}2^{2(\alpha\kone+\beta\kinf)}\|f_k\|_2^2.
\ee
\end{prop}


\begin{proof} 
{\em Step 1}. Let $0<\tilde{\alpha}<\alpha$ and $\talpha+\beta >0$.
We claim that there exists a constant $c$ such that 
\be \label{eq:normeq1} 
2^{\talpha\lone+\beta\kinf}\deltal \le c\,  
 \Big(\sum_{\substack{k_i\geq {\ell}_i\\i=1,\cdots,d}}2^{2(\alpha\kone+\beta\kinf)}\|f_k\|_2^2\Big)^{\frac{1}{2}} 
\ee
holds for all $\ell\in\N_0^d$.
Clearly,  $\delta_{\ell}: ~L_2 (\tor^d) \to L_2 (\tor^d)$ is an orthogonal projection.
The projection properties of the operator $\delta_\ell$ together with $f_k\in \mathcal{T}^k$ yields
\be\label{eq:eq3}
\deltal\leq \sum_{\substack{k_i\geq {\ell}_i\\i=1,\cdots,d}}\|\delta_{\ell}(f_k)\|_2.
\ee
Thanks to $\|\, \delta_{\ell}\, | L_2 (\tor^d) \to L_2 (\tor^d)\| = 1$
we conclude
\be\label{eq:eq4}
\deltal \leq 
\sum_{\substack{k_i\geq\ell_i\\i=1,\cdots,d}}\|f_k\|_2.
\ee
H\"older's inequality yields
\be
\deltal \leq \Big(\sum_{\substack{k_i\geq \ell_i\\i=1,\cdots,d}}
2^{-2(\talpha\kone+\beta\kinf)}\Big)^{\frac{1}{2}} \Big(\sum_{\substack{k_i\geq \ell_i\\i=1,\cdots,d}}2^{2(\talpha\kone+\beta\kinf)}\|f_k\|^2_2\Big)^{\frac{1}{2}}.\label{eq:eq5}
\ee
Now we apply Lemma \ref{lem:monotonicity} and find
\[
\sum_{\substack{k_i\geq \ell_i\\i=1,\cdots,d}}2^{-2(\talpha\kone+\beta\kinf)}
\leq 2^{-2(\talpha\lone+\beta\linf)}\sum_{\substack{k_i\geq \ell_i\\i=1,\cdots,d}}2^{-2\varepsilon|k-\ell|_1}
\lesssim 2^{-2(\talpha\lone+\beta\linf)}\, .
\]
This proves \eqref{eq:normeq1}.\\
{\em Step 2}.  Inequality \eqref{eq:normeq1} yields 
\beqq
\sum_{\ell\in \N^d_0}2^{2(\alpha\lone+\beta\linf)}\deltal^2 
&\lesssim& \sum_{\ell\in \N^d_0}2^{2(\alpha-\talpha)\lone}
\sum_{\substack{k_i\geq \ell_i\\i=1,\cdots,d}}2^{2(\talpha\kone+\beta\kinf)}\|f_k\|^2_2\\
&=& \sum_{k\in \N^d_0}2^{2(\talpha\kone+\beta\kinf)}\|f_k\|^2_2
\sum_{\substack{ \ell_i\leq k_i\\i=1,\cdots,d}}2^{2(\alpha-\talpha)\lone}\\
&\lesssim & \sum_{k\in\N_0^d}2^{2(\alpha\kone+\beta\kinf)}\|f_k\|_2^2.
\eeqq
Since the left-hand side coincides with $\|\, f\, \|_{\hab}^2$ Proposition \ref{lem:uppernorm} is proved.
\end{proof}


After one more notation we are ready for the main result of this section.

\begin{defi}
Let $\min\{\alpha,\alpha+\beta\}>\frac{1}{2}$. We define 
$$\fhab^+ :=\Big(\sum_{k\in\N^d_0} \, 2^{2(\alpha\kone+\beta\kinf)}\, \|q_k(f)\|_2^2\Big)^{\frac{1}{2}}$$
for all $f\in\hab$.
\end{defi}

\begin{satz}\label{satz:equivnorm}
Let  $\min\{\alpha,\alpha+\beta\}>\frac{1}{2}$. 
Then a function $f$ on $\TTd$ belongs to 
the space $\hab$, if and only if 
$f$ can be represented by the series \eqref{Representation} converging in $\hab$ and satisfying the  condition
\eqref{ConvegenceCondition[hab]}.
Moreover, the  norm $\|f\|_{\hab}$ is equivalent to  the norm $\fhab^+$.
\end{satz}

\begin{proof}
This result is an easy consequence of Proposition\ \ref{lem:lowernorm} and Proposition\ \ref{lem:uppernorm}, applied with $f_k=q_k(f)$.
\end{proof}

\begin{rem}\label{identisch}
 \rm
 
(i) The restriction $\min\{\alpha,\alpha+\beta\}>\frac{1}{2}$ is essentially optimal, see Remark \ref{optimal}.\\
(ii) The potential of sampling representations has been first recognized by D\~ung \cite{Di11,Di13}.
There the non-periodic situation in connection with tensor product B-spline series is treated in the unit cube.
\end{rem}





\section{Sampling on energy-norm based sparse grids}
\label{sample2}


In this section we consider the quality of approximation by sampling operators 
using energy-norm based sparse grids. In fact, a suitable sampling operator $Q_{\Delta}$ uses a slightly
larger set $\Delta_{\eps}$ compared to $\Delta$ from \eqref{f08} with the same combinatorial properties, see Lemma \ref{rank1} below. We put
\be\label{eq101} 
    \Delta_{\eps}(\xi) :=\{k\in \N_0^d:(\alpha-\eps)|k|_1-(\gamma-\beta-\eps)|k|_{\infty} \leq \xi\}\quad,\quad \xi>0\,,
\ee

\begin{satz}\label{satz:approxhab}
Let $\alpha>0$, $\gamma\geq 0$ and $\beta<\gamma$ such that $\min\{\alpha,\alpha+\beta\}>\frac{1}{2}$. Let further $0<\varepsilon<\gamma-\beta<\alpha$. 
Then there exists a constant $C=C(\alpha,\beta,\gamma,\eps,d)>0$ such that
\be 
\|f-Q_{\Delta_{\varepsilon}(\xi)}f\|_{\hiso}\leq C \, 2^{-\xi}\, \|f\|_{\hab}
\ee
holds for all $f\in \hab$ and all $\xi>0$. 
\end{satz}

\begin{proof}
{\em Step 1}. The triangle inequality in $H^{\gamma}(\tor^d)$, Lemma \ref{lem:bernstein}, and afterwards 
H\"older's inequality yield
\beqq
\|f-Q_{\Delta_{\varepsilon}(\xi)}f\|_{\hiso} & = & \Big\|\sum_{k\notin \Delta_\varepsilon(\xi)}q_k(f)\Big\|_{\hiso}
\leq \sum_{k\notin \Delta_\varepsilon(\xi)}\|q_k(f)\|_{\hiso}
\\
&\leq& \sum_{k\notin \Delta_\varepsilon(\xi)}2^{\gamma \kinf}\|q_k(f)\|_{2}\\
&=&\sum_{k\notin \Delta_\varepsilon(\xi)}2^{\alpha \kone+\beta\kinf}
2^{-(\alpha \kone+\beta\kinf)}2^{\gamma \kinf}\|q_k(f)\|_{2}
\\
&\leq & \Big(\sum_{k\notin \Delta_\varepsilon(\xi)}
2^{-2\alpha\kone+2(\gamma-\beta)\kinf}\Big)^{\frac{1}{2}}
\Big(\sum_{k\notin \Delta_\varepsilon(\xi)} \, 2^{2 (\alpha \kone+\beta\kinf)}\, \|q_k(f)\|^2_{2}\Big)^{\frac{1}{2}}.
\eeqq
Applying Theorem \ref{satz:equivnorm} we have
\beqq
\Big(\sum_{k\notin \Delta_\varepsilon(\xi)}2^{2(\alpha \kone+\beta\kinf)}\|q_k(f)\|^2_{2}\Big)^{\frac{1}{2}}\leq \fhab.
\eeqq
Consequently, we obtain the following inequality
\be\label{eq:43333}
\|f-Q_{\Delta_{\varepsilon}(\xi)}f\|_{\hiso}\leq \Big(\sum_{k\notin \Delta_\varepsilon(\xi)}2^{-2\alpha\kone+2(\gamma-\beta)\kinf}\Big)^{\frac{1}{2}} \fhab.
\ee
{\em Step 2.} Now we consider the sum
\beqq
\sum_{k\notin \Delta_\varepsilon(\xi)}2^{-2\alpha\kone+2(\gamma-\beta)\kinf}\leq \sum_{i=1}^d\sum_{\substack{k\notin \Delta_\varepsilon(\xi)\\k\in K_i}}2^{-2\alpha\kone+2(\gamma-\beta)\kinf},
\eeqq
where
\be\label{eq:tka}
K_i:=\{k\in\N^d_0:k_i=\kinf\}\, \qquad i=1, \ldots \, , d.
\ee
We want to find a proper upper bound for 
$$\sum_{\substack{k\notin \Delta_\varepsilon(\xi)\\
k\in K_i}} 2^{-2\alpha\kone+2(\gamma-\beta)\kinf}.
$$ 
For simplicity we restrict ourselves to the case $i=1$ with $k_1=\kinf$ and set
\be \tk:=(k_2,\cdots,k_d).\label{eq:tk}\ee 
Indeed,
$\kone=k_1+|\tk|_1$
holds for all $k\in\N^d_0$.
So the following equivalence is true
\beqq
k\notin \Delta_\varepsilon(\xi)&\Longleftrightarrow& (\alpha-\varepsilon)\kone-((\gamma-\beta)-\varepsilon)k_1>\xi\\
&\Longleftrightarrow&(\alpha-\varepsilon)|\tk|_1+(\alpha-(\gamma-\beta))k_1>\xi\\
&\Longleftrightarrow&k_1>\frac{\xi-(\alpha-\varepsilon)|\tk|_1}{\alpha-(\gamma-\beta)}.
\eeqq
Using this equivalence we can proceed with
\beq
\sum_{\substack{k\notin \Delta_\varepsilon(\xi)\\k\in K_1}}2^{-2\alpha\kone+2(\gamma-\beta)\kinf}
&=& \sum_{\tk\in\mathbb{N}^{d-1}_0}2^{-2\alpha|\tk|_1}
\sum_{k_1>\max \Big\{|\tk|_{\infty}-1,\frac{\xi-(\alpha-\varepsilon)|\tk|_1}{\alpha-(\gamma-\beta)}\Big\}}
2^{-2\alpha k_1+2(\gamma-\beta) k_1}
\nonumber\\
&=&\sum_{\tk\in I_1}2^{-2\alpha|\tk|_1}\sum_{k_1\geq |\tk|_{\infty}}2^{-2\alpha k_1+2(\gamma-\beta) k_1}\label{eq:sum1}\\
&+&\sum_{\tk\notin I_1}2^{-2\alpha|\tk|_1}\sum_{k_1>\frac{\xi-(\alpha-\varepsilon)|\tk|_1}{\alpha-(\gamma-\beta)}}2^{-2\alpha k_1+2(\gamma-\beta) k_1},\label{eq:sum2}
\eeq
where
$$I_1=\{\tk\in\N^{d-1}_0:\;\frac{\xi-(\alpha-\varepsilon)|\tk|_1}{\alpha-(\gamma-\beta)}<|\tk|_{\infty}\}.$$
First we compute an upper bound for the sum in \eqref{eq:sum1}. Because of 
\[
2^{2((\gamma-\beta)-\alpha)|\tk|_\infty}\leq 2^{-2(\xi-[\alpha-\varepsilon])|\tk|_1} \qquad \mbox{if}
\qquad  \tk\in I_1\, , 
\]
we conclude
\beqq
\sum_{\tk\in I_1}\sum_{k_1\geq |\tk|_{\infty}}2^{-2\alpha k_1+2(\gamma-\beta) k_1}
&\leq & 
C \sum_{\tk\in I_1}2^{-2\alpha|\tk|_1}2^{2((\gamma-\beta)-\alpha)|\tk|_\infty}
\\
 & \lesssim &
\sum_{\tk\in I_1}\sum_{k_1\geq |\tk|_{\infty}}2^{-2\alpha k_1+2(\gamma-\beta)  k_1}
\\
&\lesssim&  2^{-2\xi}\sum_{\tk\in I_1}2^{-2\varepsilon|\tk|_1}\\
&\lesssim &  2^{-2\xi} \, .
\eeqq
Here the constant behind $\lesssim$ does not depend on $\xi$. 
\\
{\em Step 2}. Next, we estimate the sum in \eqref{eq:sum2}. Similarly as above we find
\beqq
\sum_{\tk\notin I_1}2^{-2\alpha|\tk|_1}\sum_{k_1>\frac{\xi-(\alpha-\varepsilon)|\tk|_1}{\alpha-(\gamma-\beta)}}2^{-2\alpha k_1+2(\gamma-\beta) k_1}&\lesssim& \sum_{\tk\notin  I_1}2^{-2\alpha|\tk|_1} 2^{-2(\xi-(\alpha-\varepsilon)|\tk|_1}\\
&\lesssim &2^{-2\xi}.
\eeqq
As a consequence we have
$$\sum_{k\notin \Delta_\varepsilon(\xi)}2^{-2\alpha\kone+2(\gamma-\beta)\kinf}\lesssim  2^{-2\xi}.$$
This together with \eqref{eq:43333} proves the claim.
\end{proof}

The previous result includes the case $\gamma = 0$. Let us state this special case seperately.

\begin{cor}\label{nochwas}
Let $\alpha>0$, $\beta<0$ such that $\alpha+\beta>\frac{1}{2}$ and $0<\varepsilon<-\beta<\alpha$. Then there is a constant $C=C(\alpha,\beta,\eps,d)>0$ such that
\be 
\|f-Q_{\Delta_{\varepsilon}(\xi)}f\|_{2}\leq C 2^{-\xi}\|f\|_{\hab}
\ee
holds for all $f\in \hab$ and $\xi>0$. 
\end{cor}

\begin{rem}
 \rm
(i) For the approximation of the embedding $I:H^{\alpha}_{\mix}(\tor^d) \to H^{\gamma}(\tor^d)$, where $\alpha>\gamma>0$, we could have used a 
simpler argument which does not require the sampling representation in Theorem \ref{satz:equivnorm} to estimate $\|f-Q_{\Delta_{\eps}(\xi)}f\|_{H^{\gamma}(\tor^d)}$\,.
In fact, we estimate 
\be\label{666}
  \begin{split}
    \|f-Q_{\Delta_{\eps}}f\|_{H^{\gamma}(\tor^d)} &\leq \Big\|\sum_{k\notin \Delta_\varepsilon(\xi)}q_k(f)\Big\|_{\hiso}
    \leq \sum_{k\notin \Delta_\varepsilon(\xi)}\|q_k(f)\|_{H^{\gamma}(\tor^d)}\\
    &\leq \sum_{k\notin \Delta_\varepsilon(\xi)}2^{\gamma|k|_{\infty}}\|q_k(f)\|_2\,.
  \end{split}
\ee
Due to the tensor product structure of the space $H^{\alpha}_{\mix}(\tor^d)$ we are allowed to use \cite[Lemma\ 5]{SU} to estimate 
$\|q_k(f)\|_2$. Indeed, it holds
\be\nonumber
  \begin{split}
   \|q_k(f)\|_2  &= \|(\eta_{k_1}\otimes ... \otimes \eta_{k_d})f\|_2 \leq \Big(\prod\limits_{j=1}^d \|\eta_{k_j}:H^{\alpha}(\tor) \to L_2(\tor)\|\Big) \|f\|_{H^{\alpha}_{\mix}(\tor^d)}\\
   &\lesssim 2^{-\alpha|k|_1}\|f\|_{H^{\alpha}_{\mix}(\tor^d)}\,.
  \end{split} 
\ee
Putting this into \eqref{666} yields 
$$
\|f-Q_{\Delta_{\eps}}f\|_{H^{\gamma}(\tor^d)} \leq \|f\|_{H^{\alpha}_{\mix}(\tor^d)}\sum\limits_{k\notin \Delta_{\eps}(\xi)}2^{-\alpha|k|_1+\gamma|k|_{\infty}}
$$
With exactly the same method as used in Step 2 of the proof of Theorem \ref{satz:approxhab} we obtain that 
$$
    \sum\limits_{k\notin \Delta_{\eps}(\xi)}2^{-\alpha|k|_1+\gamma|k|_{\infty}} \lesssim 2^{-\xi}\,,
$$
which yields
$$
    \|f-Q_{\Delta_{\eps}}f\|_{H^{\gamma}(\tor^d)}\lesssim 2^{-\xi}\|f\|_{H^{\alpha}_{\mix}(\tor^d)}\,.
$$
(ii) The method from (i) is not suitable if $\gamma = 0$. In fact, it produces a worse bound compared to the one obtained in Theorem \ref{beta02} below, namely
$$
      \|f-Q_{\Delta(\alpha m)}f\|_2 \lesssim 2^{-m\alpha}m^{d-1}\|f\|_{H^{\alpha}_{\mix}(\tor^d)}\,.
$$
This is actually the strategy used in \cite{T85} to obtain \eqref{f01}, see also \cite{BG04}\,.\\
\noindent (iii) Estimates of sampling operators of Smolyak-type with respect to the embeddings 
$I:H^{\alpha}_\mix ([0,1]^d) \to H^\gamma ([0,1]^d)$ may be found also in the papers
\cite{BG99,BG04,Gr,SST08} and the recent one \cite{GH14}. In particular, Bungartz and Griebel have used 
energy-norm based sparse grids in case $\alpha = 2$ and $\gamma =1$. These authors have taken care of the dependence of all constants on 
the dimension $d$, an important  problem in high-dimensional approximation, which we have ignored here.
\end{rem}


\section{Sampling on Smolyak grids}
\label{smolyak}

In this section we intend to apply our new method to situations where the classical Smolyak algorithm is used. 
On the one hand we give shorter proofs for existing results and extend some of them concerning the used approximating operators on the other hand. 

\subsection{The mixed-mixed case}

\label{mixmix2}

We consider sampling operators for functions in $H^{\alpha}_{\mix}(\tor^d)$ measuring the error in $\hmixedga$. The associated operator $Q_{\Delta}$
is this time given by 
\be\label{667}
    \Delta(\xi) = \Delta(\alpha,\gamma;\xi) := \{k\in \N_0^d~:~(\alpha-\gamma)|k|_1 \leq \xi\}\quad,\quad \xi>0\,.
\ee

\begin{satz}\label{mixmix}
Let $\gamma>0$ and $\alpha>\max\{\gamma,1/2\}$. Then there is a constant $C=C(\alpha,\gamma,d)>0$ such that
$$\| f-Q_{\Delta(\xi)}f\|_{\hmixedga}\leq C 2^{-\xi}\fmixed$$
holds for all $f\in\hmixed$ and $\xi>0$.
\end{satz}

\begin{proof}
We employ Proposition\ \ref{lem:uppernorm} to $H^{\gamma}_{\mix}(\tor^d)$ with the sequence $(f_k)_{k\in \N_0^d}$ given by 
$$
    f_k = \left\{\begin{array}{rcl}
              q_k(f)&:&k\notin \Delta(\xi),\\
              0&:&k\in \Delta(\xi)\,.
          \end{array}\right.
$$
Note, that the only restriction for Proposition\ \ref{lem:uppernorm} is $\gamma>0$\,. Clearly,
$f-Q_{\Delta(\xi)}f = \sum\limits_{k\in \N_0^d} f_k$ and hence 
\beqq
\|f-Q_{\Delta(\xi)}f\|^2_{\hmixedga}&\lesssim&\sum_{k \in \N_0^d}2^{2\gamma|k|_1}\|f_k\|^2_2\\
&=& \sum_{k\notin \Delta(\xi)}2^{2(\gamma-\alpha)\kone}2^{2\alpha\kone}\|q_k(f)\|^2_2\\
&\leq& 2^{-2\xi}\sum_{k\in \N_0^d}2^{2\alpha\kone}\|q_{k}(f)\|^2_2.
\eeqq
Applying Theorem \ref{satz:equivnorm} (here we need $\alpha>1/2$) completes the proof since
\beqq
\sum_{k\in \N_0^d}2^{2\alpha\kone}\|q_{k}(f)\|^2_2&\lesssim&
\fmixed^2.
\eeqq
\end{proof}

As a direct consequence of Theorem \ref{mixmix}, we obtain the following result for the weaker error norm $\|\cdot\|_{H^{\gamma}(\tor^d)}$

\begin{cor}
Let $\alpha>\frac{1}{2}$ and $0<\gamma<\alpha$. Then there is a constant $C=C(\alpha,\gamma,d)>0$ such that
\be 
\|f-Q_{\Delta(\xi)}f\|_{\hiso}\leq C 2^{-\xi}\|f\|_{\hmixed}
\ee
holds for all $f\in \hmixed$ and $\xi>0$. 
\end{cor}

\begin{rem}
 \rm
Sampling with Smolyak operators has some history. Closest to us are  
Temlyakov \cite{T85,T93,T93b} and D\~ung \cite{Di90}-\cite{Di13}, see also \cite{si02}, \cite{si06} and \cite{SU}.
In almost all contributions preference was given to situations where the target space was $L_q (\tor^d)$. Let us also 
refer to the recent preprint \cite{GH14}.
\end{rem}



\subsection{The case $\alpha>\gamma-\beta= 0$}


Now we are interested in the embedding
$$
    I:H^{\alpha,\beta}(\tor^d) \to H^{\beta}(\tor^d)\,.
$$
The sampling operator $Q_{\Delta(\xi)}$ is determined by $\Delta(\xi)$ from \eqref{f012b}\,. Let us simplify the structure by considering the
index sets $\Delta(\alpha m)$ for $m\in \N$ which consists of all $k\in \N_0^d$ satisfying $|k|_1 \leq m$.

\begin{satz}\label{beta02}
Let $\beta=\gamma\geq 0$  and  $\alpha>\frac{1}{2}$. Then there is a constant $C=C(\alpha,\beta,d)>0$ such that
$$\| f-Q_{\Delta(\alpha m)}f\|_{H^\beta(\tor^d)}\leq C 2^{-m\alpha}m^{\frac{d-1}{2}}\fhab$$
holds for all $f\in\hab$ and $m\in\N$.
\end{satz}

\begin{proof}
We proceed as in proof of Theorem \ref{satz:approxhab}.
The triangle inequality in $H^{\beta}(\tor^d)$ yields
\[
\|f-Q_{\Delta(\alpha m)}f\|_{\hisobeta}=\Big\|\sum_{k\notin \Delta(\alpha m)}q_k(f)\Big\|_{\hisobeta}
\leq\sum_{k\notin \Delta(\alpha m)}\|q_k(f)\|_{\hisobeta}\, .
\]
Applying Lemma \ref{lem:bernstein} gives
\beqq
\|f-Q_{\Delta(\alpha m)}f\|_{\hisobeta}&\lesssim& \sum_{k\notin \Delta(\alpha m)}2^{\beta\kinf}\|q_k(f)\|_{2}.
\eeqq
Proceeding with H\"older's inequality leads to
$$
\|f-Q_{\Delta(\alpha m)}f\|_{2}\leq \Big(\sum_{\kone>m}2^{-2\alpha\kone}\Big)^{\frac{1}{2}}\Big(\sum_{\kone>m}2^{2(\alpha\kone+\beta\kinf)}\|q_k(f)\|^2_2\Big)^{\frac{1}{2}}.
$$
Employing the upcoming lemma and Theorem \ref{satz:equivnorm} finishes the proof.

\end{proof}

\begin{lem}\label{lem:smolsum}
Let $\alpha>0$. Then
\be
\sum_{\kone>m}2^{-2\alpha\kone} \lesssim m^{d-1}2^{-2\alpha m} \label{eq:sum8}
\ee
holds for all $m>0$.
\end{lem}
\begin{proof}This lemma is well known. Let us prove it for completeness.
We decompose the sum in the following two parts
\be\sum_{\kone>m}2^{-2\alpha\kone}= \sum_{\substack{\kone>m\\\kinf\leq m}}2^{-2\alpha\kone}+\sum_{\kinf>m}2^{-2\alpha\kone}.\label{eq:sum3}\ee
First we compute an upper bound for the second sum in \eqref{eq:sum3}. 
Again we use the convention for $\tk$ of $k$ from \eqref{eq:tk} and decompose as follows
\beqq 
\sum_{\kinf>m}2^{-2\alpha\kone} &\leq & \sum_{i=1}^d \sum_{\substack{k_i>m\\k\in \N_0^d}}2^{-2\alpha\kone}
=d\sum_{\tk\in \N_0^{d-1}}2^{-2\alpha|\tk|_1}\sum_{k_1>m}2^{-\alpha k_1}\\
&\lesssim&2^{-\alpha m}.
\eeqq
The first sum in \eqref{eq:sum3} gives
\beqq
\sum_{\substack{\kinf\leq m\\\kone >m}}2^{-2\alpha\kone}
&\leq&\sum_{k_2=0}^m\ldots\sum_{k_d=0}^m\sum_{k_1=m-|\tk|_1}^{\infty}2^{-2\alpha\kone}
\\
&\lesssim& (m+1)^{d-1} 2^{-2\alpha m}.
\eeqq
Consequently,
\beq
\sum_{\kone>m}2^{-2\alpha\kone} &\lesssim& m^{d-1}2^{-2\alpha m} 
\eeq
holds for all $m>0$.
\end{proof}


\subsection{The case $\gamma= 0$}
\label{beta00}


From Theorem \ref{beta02} we immediately obtain the special case ($\gamma = \beta = 0$) 
$$\| f-Q_{\Delta(\alpha m)}f\|_{2}\leq C 2^{-m\alpha}m^{\frac{d-1}{2}}\|f\|_{H^{\alpha}_{\mix}(\tor^d)}\quad,\quad m\in \N\,,$$
compare with \cite{si02}, \cite{SU}. With our methods we can additionally show an error bound for $L_{\infty}(\tor^d)$ instead of $L_2(\tor^d)$. 

\begin{satz}\label{tem}
Let $\alpha>\frac{1}{2}$. Then there is a constant $C=C(\alpha,d)>0$ such that
$$\| f-Q_{\Delta(\alpha m)}f\|_{\infty}\leq C 2^{-m(\alpha-\frac{1}{2})}m^{\frac{d-1}{2}}\fmixed$$
holds for all $f\in\hmixed$ and $m\in\N$.
\end{satz}

\begin{proof} As above with Lemma \ref{lem:bernsteinnikolskij} we conclude
\beq\label{eq:upbound5}
\|f-Q_{\Delta(\alpha m)}f\|_{\infty}&=& \Big\|\sum_{k\notin \Delta(\alpha m)}q_k(f)\Big\|_{\infty}
 \leq \sum_{k\notin \Delta(\alpha m)}\|q_k(f)\|_{\infty}
\nonumber
 \\
&\leq&\sum_{\kone>m}2^{\kone/2}2^{\alpha\kone}2^{-\alpha\kone}\|q_k(f)\|_2
\nonumber
\\
& \le & \Big(\sum_{\kone>m} 2^{-2\kone(\alpha-\frac{1}{2})}\Big)^{\frac{1}{2}}
\Big(\sum_{\kone>m}2^{2\alpha\kone}\|q_k(f)\|^2_2\Big)^{\frac{1}{2}}.
\eeq
Applying Lemma \ref{lem:smolsum} and Theorem \ref{satz:equivnorm} proves the claim.
\end{proof}

Now we turn to the case $2 < q < \infty$. The following result allows for comparing the present situation with the results in Subsection \ref{mixmix2}.
\begin{lem}\label{lem:qineq}
Let $2 < q< \infty$. Then
$$\|f\|_q\lesssim \Big(\sum_{k\in\N_0^d}\|\delta_k(f)\|^2_q\Big)^{1/2} \lesssim \Big(\sum_{k\in\N_0^d}2^{2|k|_1(1/2-1/q)}\|\delta_k(f)\|^2_2\Big)^{1/2}=\|f\|_{H^{\frac{1}{2}-\frac{1}{q}}_{\mix}(\tor^d)}$$ 
holds true for any $f\in L_q(\tor^d)$, where the right-hand side may be infinite. 
\end{lem}

\begin{proof}
 \rm
The proof of the first relation in Lemma \ref{lem:qineq} is elementary using the Littlewood-Paley decomposition in $L_q(\tor^d)$ together with $q/2\geq 1$, see for instance \cite[Theorem\ 0.3.2, Page 20]{T93b}. The second relation follows by an application of Nikol'skij's inequality in Lemma \ref{lem:bernsteinnikolskij}. 
\end{proof}
\begin{rem}
Let us mention that Lemma \ref{lem:qineq} can be 
refined to 
$$
    \|f\|_q  \lesssim \Big(\sum_{k\in\N_0^d}2^{q|k|_1(1/2-1/q)}\|\delta_k(f)\|^q_2\Big)^{1/q}\,.
$$
For this deep result we refer to \cite[Lemma\ II.2.1]{T93b} and to \cite[Lemma\ 5.3]{Di11} as well as \cite[Lemma\ 1]{SS04} for non-periodic versions. In a more general context this embedding is a special case of a Jawerth/Franke type embedding, see \cite{HaVy09}.
\end{rem}


\section{Sampling numbers}
\label{samp}


In this section we will restate the approximation results from Sections \ref{sample2} and \ref{smolyak} in terms of the 
number of degrees of freedom. We additionally show the asymptotic optimality with regard to sampling numbers of the sampling operators considered in Sections \ref{sample2} and \ref{smolyak}. This
requires estimates of the rank of the corresponding sampling operators. A lower bound for the rank is deduced from the 
fact that the respective sampling operators reproduce trigonometric polynomials from modified hyperbolic crosses 
$\mathcal{H}_{\Delta}$. Recall that our approximation scheme is based on the classical trigonometric interpolation. 
We have used several times the fact that the operator $I_m$ defined in \eqref{f015} reproduces univariate 
trigonometric polynomials of degree less than or equal to $m$.
What concerns the operator $Q_\Delta$ in \eqref{f06} we can prove the following general reproduction result. 

\begin{lem}\label{lem:trigreprod} Let $\Delta\subset \N_0^d$ be a solid finite set meaning that  
$k\in \Delta$ and $\ell \leq k$ implies $\ell \in \Delta$. Then
$Q_{\Delta}$ reproduces trigonometric polynomials with frequencies in 
\begin{equation}\label{hy}
\mathcal{H}_{\Delta}:=\bigcup_{k\in \Delta}\mathcal{P}_k\,,
\end{equation}
where $\mathcal{P}_k$ is defined in \eqref{f014}.
\end{lem}
\begin{proof} We follow the arguments in the proof of \cite[Lemma\ 1]{SU}. By the fact that $|\Delta|<\infty$ we find a $m\geq 0$ such that $$\Delta \subset \{0,\ldots,m\}^d.$$
Let $$T:=\sum_{|k|_{\infty}\leq m}\bigotimes_{i=1}^d \eta_{k_i}\quad\mbox{and}\quad R:=\sum_{\substack{k\notin \Delta\\|k|_{\infty}\leq m}}\bigotimes_{i=1}^d \eta_{k_i}.$$
Of course, it holds
$$Q_{\Delta}=T-R\,.$$
Since
$$\sum_{k=0}^{m}\eta_k=I_{2^{m}}$$ 
we obtain
$$T=\bigotimes_{i=1}^d I_{2^{m}}.$$
Obviously, for $\ell \in \mathcal{H}_{\Delta}$ the univariate reproduction property yields 
$$(Te^{i\ell\cdot})(x)=\prod_{j=1}^d(I_{2^m}e^{i\ell_j\cdot})(x_j)=e^{i\ell x}$$
for all $x\in \tor^d$. It remains to prove $Re^{i\ell\cdot}\equiv 0$. Let $k=(k_1,\ldots,k_d)\in \N_0^d$ such that $k\notin \Delta$. Due to $\ell\in H_{\Delta}$ there exists $u\in \Delta$ with $|\ell_i|\leq 2^{u_i}$ for all  $i=1,\ldots,d$.
The solidity property of $\Delta$ yields the existence of $j\in\{1,\ldots,d\}$ with
$$u_j<k_j.$$
This gives $$|\ell_j|\leq 2^{u_j}\leq 2^{k_j-1}<2^{k_j}.$$
Finally, by the univariate reproduction property, we obtain
$$\eta_{k_j}(e^{i\ell_j \cdot}) = (I_{2^{k_j}}-I_{2^{k_j-1}})e^{i\ell_j \cdot}=0.$$
\end{proof}
\noindent The previous result immediately implies the relation
$$
    \rank Q_{\Delta} \geq \sum\limits_{k\in \Delta} 2^{|k|_1}
$$
if $\Delta\subset \N_0^d$ is solid.
\begin{lem}Let $\alpha>0$, $\gamma \geq 0$ and $\beta<\gamma$ such that $0<\gamma-\beta\leq \alpha$
\begin{enumerate}
\item The index sets $\Delta(\alpha,\beta,\gamma;\xi)$ defined in \eqref{f08} and 
$\Delta(\eps, \alpha,\beta,\gamma;\xi)$ defined in \eqref{f082} are solid sets in the sense of Lemma\ \ref{lem:trigreprod} for every $\xi> 0$.

\item The index set  $\Delta(\alpha; \xi)$ defined in \eqref{f012b} is a solid set for every $\xi>0$.
\end{enumerate}
\end{lem}
\begin{proof}
The second result is trivial. We prove the first one. Let
$$\psi(k):=\alpha\kone-(\gamma-\beta)\kinf.$$
The set $\Delta(\xi)$ consists of all $k\in\N_0^d$ with $\psi(k)\leq \xi$. Applying Lemma \ref{lem:monotonicity} yields
$$\psi(k')\leq \psi(k)\leq \xi$$
for all $k'\leq k\in \Delta(\xi)$. That means all the $k'$ also belong to $\Delta(\xi)$.
\end{proof}

\begin{rem}
 Hyperbolic crosses $\mathcal{H}_{\Delta(\xi)}$ (with $\Delta(\xi)$ from \eqref{f08} and \eqref{f012b}) in the 2-plane:\\
\begin{minipage}[hbt]{7.5cm}
\begin{flushleft}
\begin{figure}[H]\label{fig:3}
\centering

\caption{$\alpha=2,\;\beta=0,\;\gamma=1,\;\xi=4$}
\label{fig_Squares}
\end{figure}\end{flushleft}
\end{minipage}
\begin{minipage}[hbt]{8cm}
\begin{flushright}
\begin{figure}[H]
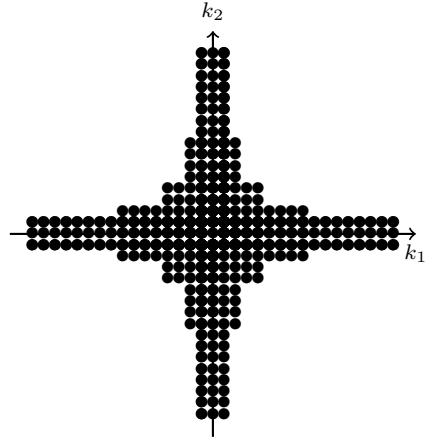
\label{fig:4}
\centering
%
\caption{$\alpha=1,\;\xi=4$}
\label{fig_Squares}
\end{figure}
\end{flushright}
\end{minipage}
  \\\noindent Comparing Figure 3 and 4 shows that energy norm based hyperbolic crosses contain ``mostly'' anisotropic building blocks than its classical (Smolyak) counterpart.
\end{rem}
In the next lemma we give sharp estimates for $\sum_{k\in \Delta (\xi)}2^{\kone}$ with $\Delta(\xi)$
from \eqref{f08}. 

\begin{lem}\label{rank1}
Let $\alpha>0$, $\gamma\geq 0$, $\beta\in \re$ such that $\gamma>\beta$ and $\alpha>\gamma-\beta$. Then
$$\sum_{k\in \Delta (\xi)}2^{\kone}\asymp 2^{\frac{\xi}{\alpha-(\gamma-\beta)}}$$
holds for all $\xi\geq\alpha-(\gamma-\beta)$, where the constants behind ``$\asymp$'' only depend on $\alpha$, $\gamma-\beta$, and $d$\,.
\end{lem}

\begin{proof}
{\em Step 1.} First we deal with the upper bound. 
We are going to use the same notation as in \eqref{eq:tka} and \eqref{eq:tk}.
We obtain the following inequality
\[
\sum_{k\in \Delta(\xi)} \, 2^{\kone} \leq 
\sum_{i=1}^d \sum_{k\in K_i \cap \Delta(\xi)}\, 2^{\kone}.
\] 
By symmetry it will be enough to deal with $i=1$.
Hence
\beqq
\sum_{k\in \Delta(\xi)} 2^{\kone} &\leq& d \, 
\sum_{k\in K_1 \cap \Delta(\xi)}\, 2^{\kone}.
\eeqq 
Now we want to decompose the summation over k. Since  $k_1\geq |k|_\infty$ we find
\beqq
k\in \Delta(\xi) &\Longleftrightarrow& 
\alpha \kone- (\gamma-\beta) k_1 \leq \xi\\
&\Longleftrightarrow&\alpha (|\tk|_1+k_1)- (\gamma-\beta) k_1 \leq \xi\\
&\Longleftrightarrow& k_1 \leq \frac{\xi-\alpha|\tk|_1}{\alpha-(\gamma-\beta)} \, .
\eeqq
This implies 
\beqq
|\tk|_\infty \leq \frac{\xi-\alpha|\tk|_1}{\alpha-(\gamma-\beta)} &\Longleftrightarrow& \alpha |\tk|_1+
(\alpha-(\gamma-\beta)) |\tk|_{\infty}\leq\xi.
\eeqq
We shall use these inequalities to produce an appropriate  
decomposition of $K_1 \cap \Delta(\xi)$ which results in
\beqq
\sum_{k\in \Delta(\xi)}2^{\kone} &\leq& d 
\sum_{\substack{\tk \in \N_0^{d-1}\\ \alpha |\tk|_1+
(\alpha-(\gamma-\beta)) |\tk|_{\infty}\leq\xi}} 2^{|\tk|_1}
\sum_{k_1=|\tk|_\infty}^{\frac{\xi-\alpha|\tk|_1}{\alpha-(\gamma-\beta)}}2^{k_1}
\\
&\lesssim & 2^{\frac{\xi}{\alpha-(\gamma-\beta)}}
\sum_{\alpha |\tk|_1+(\alpha-(\gamma-\beta)) 
|\tk|_{\infty}\leq\xi} 2^{\frac{-\tilde{\alpha}}{\alpha-(\gamma-\beta)}|\tk|_1}
\\
& \lesssim &2^{\frac{\xi}{\alpha-(\gamma-\beta)}}\, , 
\eeqq
since $\alpha/(\alpha-(\gamma-\beta)>0$.
\\
{\em Step 2}. We prove the lower bound. First we claim  that 
$$k^*:=\Big\lfloor\frac{\xi}{\alpha-(\gamma-\beta)}\Big\rfloor(1,0,\cdots,0)\in 
\Delta(\xi).$$ 
Indeed, 
\beqq 
k^* \in \Delta(\xi)& \Longleftrightarrow & 
\alpha|k^*|_1-(\gamma-\beta)|k^*|_{\infty}\leq \xi\\
&\Longleftrightarrow& 
(\alpha-\varepsilon)\Big\lfloor\frac{\xi}{\alpha-(\gamma-\beta)}\Big\rfloor-
((\gamma-\beta)-\varepsilon)\Big\lfloor\frac{\xi}{\alpha-(\gamma-\beta)}\Big\rfloor\leq \xi
\\
&\Longleftrightarrow& 
(\alpha-(\gamma-\beta))\Big\lfloor\frac{\xi}{\alpha-(\gamma-\beta)}\Big\rfloor\leq \xi.
\eeqq
Obviously, the last inequality is true. Consequently 
\beqq
\sum_{k\in \Delta(\xi)}2^{\kone} \geq 
2^{|k^*|_1}=2^{\Big\lfloor\frac{\xi}{\alpha-(\gamma-\beta)}\Big\rfloor}
\geq 2^{\frac{\xi}{\alpha-(\gamma-\beta)}-1}.
\eeqq
The proof is complete.
\end{proof}

\begin{cor}\label{cor1}
Let $\alpha>0$, $\gamma\geq 0$, $\beta\in \re$ such that $\gamma>\beta$ and $\alpha>\gamma-\beta$. Let further  $\Delta(\xi)$ as in \eqref{f08}.
\\
{\rm (i)} The sampling operator 
$Q_{\Delta(\xi)}$ uses at most $C2^{\frac{\xi}{\alpha-(\gamma-\beta)}}$ function values, where the constant $C>0$ only depends on $\alpha$, $\gamma-\beta$ and $d$\,.
\\
{\rm (ii)} The rank of  the linear  operator 
$Q_{\Delta(\xi)}$ satisfies 
\[ 
\rank Q_{\Delta(\xi)} \asymp 2^{\frac{\xi}{\alpha-(\gamma-\beta)}}\, , \qquad 
\xi\geq\alpha-(\gamma-\beta)\, ,
\]
where the constants behind ``$\asymp$'' only depend on $\alpha$, $\gamma-\beta$, and $d$\,.
\end{cor}

\begin{proof}
 Clearly, $I_m f$ uses $2m+1$ values of  function $f$, hence $\eta_m f$ is using $\le 2^{m+2}$ function values.
This implies that 
$q_k f$ applies $\le 2^{2d}\, 2^{|k|_1}$ function values. As a consequence of Lemma \ref{rank1}
we find that $Q_{\Delta(\xi)} f$ is using
\[
\lesssim \sum_{k\in \Delta(\xi)}2^{\kone}\asymp 2^{\frac{\xi}{\alpha-(\gamma-\beta)}}
\]
function values of $f$. Part (ii) follows from Lemma \ref{lem:trigreprod} and the lower bound in Lemma \ref{rank1}.
\end{proof}
Let us now count the degree of freedom for a classical Smolyak grid.

\begin{lem}\label{rank2}
For any $d\in \N$ and $m\in\N_0^d$, we have the inequality
$$\Big(\frac{m+d-1}{d-1}\Big)^{d-1}2^m \leq \sum_{|k|_1 \leq m}2^{\kone}\leq \Big[\frac{e(m+d-1)}{d-1}\Big]^{d-1}2^{m+1}.$$
\end{lem}

\begin{proof} This assertion is a direct consequence of \cite[Lemma\ 3.10]{DiUl13} together with the well-known relation 
$$\Big(\frac{N}{n}\Big)^n\leq \binom{N}{n}\leq \Big(\frac{eN}{n}\Big)^n\,.$$
\end{proof}

\begin{cor}\label{cor2}
Let $m\in \N$ and 
$$
    \Delta = \{k\in \N_0^d~:~|k|_1 \leq m\}\,.
$$
{\rm (i)}  The sampling operator $Q_{\Delta}$ is using at most $C m^{d-1}\, 2^{m}$ function values, where $C$ decays super-exponentially in $d$.
\\
{\rm (ii)} The rank of  the linear  operator 
$Q_{\Delta}$ satisfies 
\[ 
\rank Q_{\Delta} \asymp m^{d-1}\, 2^{m}\, , \qquad 
m \in \N\, .
\]
\end{cor}

\begin{proof}
Part (i) follows from  the fact that $q_k(f)$ uses $2^{2d}2^{|k|_1}$ function values for any $k$ together with the upper bound in Lemma \ref{rank2}. The second assertion can be derived by using the reproduction properties of 
$Q_{\Delta}$, see Lemma \ref{lem:trigreprod}, and the lower bound in Lemma \ref{rank2}.
\end{proof}
\begin{rem}For $d=2$ the sampling grids of $Q_{\Delta(\xi)}$ as in \eqref{f08} and \eqref{f012b} look like:\\
\begin{minipage}[hbt]{7cm}
\begin{flushleft}
\begin{figure}[H]
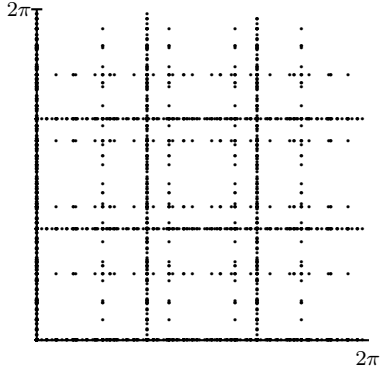

\centering

\caption{$\alpha=2,\;\beta=0,\;\gamma=1,\;\xi=5$}
\label{fig_Squares}
\end{figure}
\end{flushleft}
\end{minipage}
\begin{minipage}[hbt]{7cm}
\begin{flushright}
\begin{figure}[H]
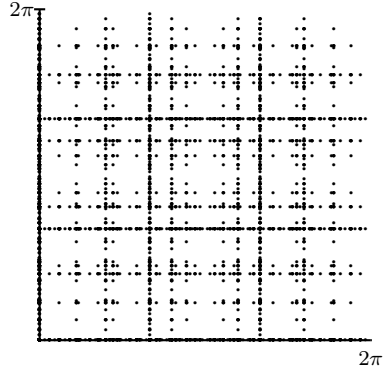

\centering
%
\caption{$\alpha=1,\;\xi=5$}
\label{fig_Squares}
\end{figure}
\end{flushright}
\end{minipage}
\end{rem}
\noindent These figures show that the point sets of $Q_{\Delta(\xi)}$ have a lot of internal structure. However, they are far from being uniformly distributed within $\T$.\\ \\
\noindent Now we are in position to formulate our results in terms of sampling numbers.

\begin{satz}
 \label{main4}
 Let $\alpha, \beta \, , \gamma\,  \in \re$ 
 such that $\min\{\alpha, \alpha + \beta\} >1/2$, $\gamma \ge 0$ and $0 <  \gamma - \beta < \alpha$.
Then it holds
\be\label{ws-11}
g_m(I_1:\hab \to H^\gamma (\tor^d)) \asymp 
a_m(I_1:\hab \to H^\gamma (\tor^d)) \asymp 
m^{-(\alpha-\gamma + \beta)}  \, , \qquad m \ge 1\, .  
\ee
\end{satz}

\begin{proof}
Proposition\ \ref{approx} below shows
\[
m^{-(\alpha -(\gamma - \beta))} \lesssim a_m(I_1:\hab \to H^\gamma (\tor^d)) \le 
g_m(I_1:\hab \to H^\gamma (\tor^d))\, , \qquad m \in \N\, . 
\]
Suppose $0 < \varepsilon < \gamma - \beta$. Let $D_\varepsilon (\xi)$ be the number of function values the operator 
$\Delta_{\varepsilon}(\xi)$ is using.
Then Theorem \ref{satz:approxhab} yields
$$
\|f-Q_{\Delta_{\varepsilon}(\xi)}f\|_{\hiso}  \lesssim  
\Big( \frac{D_\varepsilon(\xi)}{2^{\xi/(\alpha -(\gamma - \beta)) }}\Big)^{\alpha -(\gamma - \beta)}\, 
D_\varepsilon(\xi)^{-(\alpha -(\gamma - \beta))}\, \|f\|_{\hab}\,.
$$
Applying Corollary \ref{cor1}, (i) with $\alpha-\eps$ and $\gamma-\beta-\eps$ shows that 
$$
   \frac{D_\varepsilon(\xi)}{2^{\xi/(\alpha -(\gamma - \beta)) }} \leq C_1(\eps,\alpha,\gamma-\beta,d)\,.
$$
This proves the estimate from above in case $m= D_\varepsilon$. The corresponding estimate for all $m$ follows by a simple monotonicity argument. 
\end{proof}

\begin{rem} \rm
In case $\beta=0$ Griebel and Hamaekers recently proved a similar upper bound for $g_m(I_1)$ (see \cite[Lemma 9]{GH14}). Under the conditions of Theorem \ref{main4} the family of sampling operators $Q_{\Delta_{\varepsilon}(\xi)}$ for $0 < \varepsilon < \gamma - \beta$
is optimal in order.
\end{rem}

The next theorem collects sharp results for sampling numbers which are based on Smolyak's algorithm. 

\begin{satz}
 \label{main1}
Let $\alpha >1/2$ and suppose $0 < \gamma < \alpha$.\\
{\rm (i)} We have for $m\ge 2$
\be\label{ws-04}
g_m(I_5:\hmixed \to H^\gamma_{\mix} (\tor^d)) \asymp 
a_m(I_5:\hmixed \to H^\gamma_{\mix} (\tor^d)) \asymp 
m^{-(\alpha-\gamma)}(\log m)^{(d-1)(\alpha-\gamma)}\,.
\ee
{\rm (ii)} Let $2 < q < \infty$. Then we have for $m \geq 2$
\be\label{ws-05}
\begin{split}
   g_m(I_4:\hmixed \to L_q (\tor^d)) &\asymp a_m(I_4:\hmixed\to L_q (\tor^d))\\ 
   &\asymp m^{-\alpha + \frac 12 - \frac 1q} \, (\log m)^{(d-1)\, (\alpha - \frac 12 + \frac 1q)} \,.
\end{split}
\ee
\\
{\rm (iii)} In case  $q = \infty$ it holds for all $m\geq 2$
\be\label{ws-05b}
g_m(I_4:\hmixed \to L_\infty (\tor^d)) \asymp 
a_m(I_4:\hmixed \to L_\infty (\tor^d)) \asymp
m^{-\alpha + \frac 12} \, (\log m)^{(d-1)\, \alpha} \,.
\ee
\end{satz}

\begin{proof}
{\em Proof of (i).}\\ Proposition \ref{approx}, (iii)  below shows for $m\geq 2$
\[
m^{-(\alpha-\gamma)}(\log m)^{(d-1)(\alpha-\gamma)} \lesssim a_m(I_5:\hmixed \to H^\gamma_{\mix} (\tor^d)) \le 
g_m(I_5:\hmixed \to H^\gamma (\tor^d))\,.
\]
Concerning the estimate from above we apply Theorem \ref{mixmix} with $\xi = (\alpha-\gamma)m$ for $m\in \N$. This gives 
\be\label{668}
    \|f-Q_{\Delta((\alpha-\gamma)m)}f\|_{H^{\gamma}_{\mix}(\tor^d)} \lesssim 2^{-(\alpha-\gamma)m}\quad,\quad m\in \N\,.
\ee
Let $D(m)$ be the number of function values used by $Q_{\Delta((\alpha-\gamma)m)}f$. By Theorem \ref{cor2},(i),(ii) we know that 
$$
    D(m) \asymp m^{d-1}2^m\quad\mbox{and}\quad \log D(m) \asymp \log m\,.
$$ 
Rewriting \eqref{668} gives
$$
  \|f-Q_{\Delta((\alpha-\gamma)m)}f\|_{H^{\gamma}_{\mix}(\tor^d)} \lesssim D(m)^{-(\alpha-\gamma)} (\log D(m))^{(d-1)(\alpha-\gamma)}\,.
$$
Obvious monotonicity arguments complete the proof.
\\
{\em Proof of (ii).}\\
The estimate from below for the approximation numbers is due to Romanyuk \cite{Ro}. The corresponding estimate from above for the sampling numbers
is an immediate consequence of Lemma \ref{lem:qineq} together with (i), where $\gamma = 1/2-1/q$.
\\
{\em Proof of (iii).}\\
The estimate from below for the approximation numbers is due to Temlyakov \cite{T93}. Let us mention that this lower bound is also applied by a recent general result by Cobos, K\"uhn and Sickel \cite{CKS14}. For the details we refer to Proposition \ref{approx} below. The estimate from above for sampling numbers follows from Theorem \ref{tem} combined with 
Corollary \ref{cor2},(i),(ii) in the same way as in (i). 
\end{proof}

\begin{rem}
 \rm
As we have mentioned before, not all the results in Theorem\ \ref{main1} are new. Part (iii) reproduces a result due to Temlyakov 
\cite{T93}. Note, that our methods allow for proving this result in the framework of classical trigonometric interpolation, see Theorem \ref{tem},
whereas Temlyakov had to use de la Vall\'ee-Poussin sampling operators. In any case, it is remarkable that 
Smolyak's algorithm yields optimal bounds here.   
A non-periodic version of (ii) has been proved recently in D\~{u}ng \cite{Di11}.
\end{rem}


\section{Appendix: approximation numbers}

\label{apx}


Corresponding estimates for the approximation numbers serve as a natural benchmark 
for the sampling problem we are interested in. In the sequel we mainly collect the relevant results 
from \cite{DiUl13}.

\begin{prop}\label{approx}
 {\rm (i)} Let $\alpha > \gamma -\beta >0$. Then
\[
a_n (I_1:\hab \to H^\gamma (\tor^d)) \asymp n^{-\alpha + \gamma - \beta}\, , \qquad n \in \N\, .
\]
{\rm (ii)} Let $\alpha > \gamma -\beta =0$. Then
\[
a_n (I_2: \hab\to H^\beta (\tor^d)) \asymp n^{-\alpha}\, (\log n)^{(d-1)\alpha}\, , \qquad 2\leq n \in \N\, .
\]
{\rm (iii)} Let $\alpha > \gamma \ge 0$. Then
\[
a_n (I_5: H^{\alpha}_\mix (\tor^d) \to H^\gamma_\mix (\tor^d)) \asymp n^{-(\alpha - \gamma)}\, (\log n)^{(d-1)(\alpha-\gamma)}\, , \qquad 2\leq n \in \N\, .
\]
In particular, 
\be\label{approxtem}
a_n (I_3: H^{\alpha}_\mix (\tor^d) \to L_2(\tor^d) \asymp n^{-\alpha}\, (\log n)^{\alpha(d-1)}\, , \qquad 2\leq n \in \N.
\ee
{\rm (iv)} Let $\alpha > \frac{1}{2}$. Then
\[
a_n (I_4: H^{\alpha}_\mix (\tor^d) \to L_{\infty}(\T)) \asymp n^{-\alpha+\frac{1}{2}}(\log n)^{\alpha(d-1)}\, , \qquad n \in \N\, .
\]
\end{prop}
 
\begin{proof} Let us consider (iii) first. The relation in \eqref{approxtem} is due to Temlyakov \cite[Theorem~III.4.4]{T93}. For $\gamma>0$
we use the commutative diagram
\[
\begin{CD}
H^{\alpha}_\mix (\tor^d)@>I>> H^{\gamma}_\mix (\tor^d)\\
@VAVV @AABA\\
H^{\alpha-\gamma}_\mix (\tor^d)@ >I^*>> L_2 (\tor^d)\, ,
\end{CD}
\]
where
\beqq
A f (x) & := &  \sum_{k \in \Z} c_k (f)\, \prod_{j=1}^d (1+|k_j|^2)^{\gamma/2}\, e^{ikx}\, , 
\\
B f (x) & := &  \sum_{k \in \Z} c_k (f)\, \prod_{j=1}^d (1+|k_j|^2)^{-\gamma/2}\, e^{ikx}\, .
\eeqq
Clearly, $A:~H^{\alpha}_\mix (\tor^d) \to H^{\alpha-\gamma}_\mix (\tor^d)$ and 
$B:~L_2 (\tor^d) \to H^{\gamma}_\mix (\tor^d)$ are isomorphisms.
In addition, we have $I = B \circ I^* \circ A$. The multiplicativity of the approximation numbers 
implies
\[
a_n (I:H^{\alpha}_\mix (\tor^d) \to H^\gamma_\mix (\tor^d)) \le \|\, A \, \|
\, \|\, B\, \| \, a_n (I^{*}:H^{\alpha-\gamma}_\mix (\tor^d) \to L_2 (\tor^d)) \, .
\]
Taking \eqref{approxtem} into account yields the estimate from above. For the lower bound we use the commuative diagram the other way around to see $I^* = B^{-1} \circ I \circ A^{-1}$. We obtain
\[
a_n (I^{*}:H^{\alpha-\gamma}_\mix (\tor^d) \to L_2 (\tor^d)) \le \|\, A^{-1} \, \|
\, \|\, B^{-1}\, \| \, a_n (I:H^{\alpha}_\mix (\tor^d) \to H_\mix^\gamma (\tor^d)) \, .
\]
Again \eqref{approxtem} yields (iii). To prove (ii) we use the commutative diagram
\[
\begin{CD}
H^{\alpha,\beta} (\tor^d)@>I>> H^{\beta} (\tor^d)\\
@VAVV @AABA\\
H^{\alpha}_\mix (\tor^d)@ >I^*>> L_2 (\tor^d)\, 
\end{CD}
\]
with $A,B$ modified accordingly. The result follows by \eqref{approxtem}.

The proof of (i) can be found in \cite[Theorem\ 4.7]{DiUl13}, however, with the additional restriction that $2(\gamma-\beta)>\alpha>\gamma-\beta$. For the 
convenience of the reader we give a proof without this restriction. The lower bound in (i) is a consequence of a well-known abstract result (see \cite[Theorem\ 1]{Ti60} or \cite[Theorem\ 1.4, p. 405]{LGM96}) on lower bounds for linear $n$-widths, namely

\begin{lem} \label{Tikh}
Let $L_{n+1}$ be an ${n+1}$-dimensional subspace in a Banach space $X$, and 
$B_{n+1}(r):= \{f \in L_{n+1}: \ \|f\|_X \le r\}$. Then
\[
\lambda_n(B_{n+1}(r), X)
\ \geq \
r.
\]
Here $\lambda_n(B_{n+1}(r), X)$ denotes the linear $n$-width of the set $B_{n+1}(r)$ in $X$.
\end{lem}
\noindent We apply this Lemma with $X = H^{\gamma}$ and $L_{n+1}$ to be the subspace of all trigonometric polynomials with frequencies in $\mathcal{H}_{\Delta(\xi)}$ from \eqref{hy} with $\Delta(\xi) = \Delta(\alpha,\beta,\gamma;\xi)$ and $\xi$ chosen accordingly. From Lemma \ref{rank1} we get 
$n \asymp 2^{\xi/(\alpha-(\gamma-\beta))}$. We immediately see the Bernstein type
inequality 
\begin{equation}\label{bern}
   \|f\|_{H^{\alpha,\beta}} \lesssim 2^{\xi}\|f\|_{H^{\gamma}}\quad,\quad f\in L_{n+1}\,.
\end{equation}
Hence, by choosing $r:=2^{-\xi}$ we get from \eqref{bern} that $B_{n+1}(r)$ is contained in the unit ball of $H^{\alpha,\beta}$. Finally, by 
Lemma \ref{Tikh} we conclude
$$
    a_n(I_1) \geq \lambda_n(B_{n+1}(2^{-\xi}),H^{\gamma}) =2^{-\xi} \asymp n^{-(\alpha-(\gamma-\beta))}\,.
$$
For the proof of (iv) we apply a lemma that goes back to the work of Osipenko and Parfenov (see \cite{OP95}). For more details we refer to the recent preprint by Cobos, K\"uhn and Sickel \cite{CKS14}. Plugging \eqref{approxtem} into \cite[Lemma 3.3]{CKS14} yields
\beq
a_n(I_4:H^{\alpha}_{mix}(\tor^d)\to L_{\infty}(\T)) &\geq & \frac{1}{(2\pi)^{\frac{d}{2}}}\Bigg( \sum_{j=n}^{\infty} a_j^2(I_3:H^{\alpha}_{mix}(\tor^d)\to L_2(\tor^d))\Bigg)^{\frac{1}{2}}\nonumber\\
&\gtrsim&\Bigg( \sum_{j=n}^{\infty} j^{-2\alpha}\log(j)^{2(d-1)\alpha}\Bigg)^{\frac{1}{2}}\label{eq:ainf1}.
\eeq
Estimating the sum by an integral gives
\beq
\sum_{n}^{\infty} j^{-2\alpha}(\log j) ^{2(d-1)\alpha}&\asymp&\int_{n}^{\infty} y^{-2\alpha}(\log y)^{(d-1)2\alpha}dy\nonumber\\
&\geq& (\log n)^{2(d-1)\alpha}\int_{n}^{\infty} y^{-2\alpha}dy\nonumber\\
&\asymp& (\log n)^{2(d-1)\alpha}n^{-2\alpha+1}.\label{eq:ainf2}
\eeq
Inserting \eqref{eq:ainf2} into \eqref{eq:ainf1} yields the lower bound in (iv).
\end{proof}

{~}\\
\noindent
{\bf Acknowledgments}
The authors would like to thank the organizers of the HCM-workshop ``Discrepancy, Numerical Integration, and Hyperbolic Cross Approximation'', 
where this work has been initiated, for providing a pleasant and fruitful working atmosphere. In addition, the authors would like to thank the 
Hausdorff-Center for Mathematics (HCM) and the Bonn International Graduate School (BIGS) for providing additional financial support to finish this work. 
Dinh Dung's research work is funded by Vietnam National Foundation for Science and Technology Development (NAFOSTED) under  Grant No. 102.01-2014.02, and a part of it was done when he was working as a research professor at the Vietnam Institute for Advanced Study in Mathematics (VIASM). He  would like to thank  the VIASM  for providing a fruitful research environment and working condition.

\end{document}